\documentclass{article}
\usepackage{amsthm}
\usepackage{amssymb}
\usepackage{amsfonts}
\usepackage{amsmath}
\begin{document}

\newtheorem{theorem}{Theorem}[section]
\newtheorem{definition}{Definition}[section]
\newtheorem{corollary}[theorem]{Corollary}
\newtheorem{lemma}[theorem]{Lemma}
\newtheorem{proposition}[theorem]{Proposition}
\newtheorem{step}[theorem]{Step}
\newtheorem{example}[theorem]{Example}
\newtheorem{remark}[theorem]{Remark}

\font\sixbb=msbm6
%\font\sevenbb=msbm7
\font\eightbb=msbm8
%\font\ninebb=msbm9
%\font\tenbb=msbm10
\font\twelvebb=msbm10 scaled 1095
%\font\thirteenbb=msbm10 scaled 1315
%\font\fourteenbb=msbm10 scaled \magstep2
%%%%%%%%%%%%
\newfam\bbfam
\textfont\bbfam=\twelvebb \scriptfont\bbfam=\eightbb
                           \scriptscriptfont\bbfam=\sixbb
%\textfont\bbfam=\thirteenbb \scriptfont\bbfam=\eightbb
%                            \scriptscriptfont\bbfam=\sixbb

\newcommand{\tr}{{\rm tr \,}}
\newcommand{\linspan}{{\rm span\,}}
\newcommand{\rank}{{\rm rank\,}}
\newcommand{\diag}{{\rm Diag\,}}
\newcommand{\Image}{{\rm Im\,}}
\newcommand{\Ker}{{\rm Ker\,}}
\newcommand{\U}{\mathcal{U}}
\newcommand{\W}{\mathcal{W}}

\def\bb{\fam\bbfam\twelvebb}
\newcommand{\enp}{\begin{flushright} $\Box$ \end{flushright}}

\def\cB{{\mathcal{B}}}
\def\cD{{\mathcal{D}}}
\def\cU{{\mathcal{U}}}
\def\cV{{\mathcal{V}}}

\def\C{{\mathbb{C}}}
\def\D{{\mathbb{D}}}
\def\N{{\mathbb{N}}}
\def\Q{{\mathbb{Q}}}
\def\R{{\mathbb{R}}}
\def\Z{{\mathbb{Z}}}

\title{Order embeddings of real matrix domains\thanks{The author was supported by grants N1-0368, J1-60025, and P1-0288 from ARIS, Slovenia.}}
\author{Peter \v Semrl\footnote{Institute of Mathematics, Physics, and Mechanics, Jadranska 19, SI-1000 Ljubljana, Slovenia;
Faculty of Mathematics and Physics, University of Ljubljana,
        Jadranska 19, SI-1000 Ljubljana, Slovenia,  peter.semrl@fmf.uni-lj.si}
        }

\date{}
\maketitle

\begin{abstract}
Let $n$ be a positive integer, $n \not=1$, and $S_n$ the set of all $n \times n$ real symmetric matrices.
A nonempty subset $\U \subset S_n$ is called a matrix domain if it is open and connected and a map $\phi : \U \to S_n$ is said to be an order emebedding if for every pair $X,Y \in \U$ we have $X \le Y \iff \phi (X) \le \phi(Y)$. We describe the general form of such maps.
\end{abstract}
\maketitle



\section{Introduction}

Let $n \ge 2$ be an integer and $S_n$ the set of all $n \times n$ real symmetric matrices. As usual we will identify $\mathbb{R}^n$ with the set of all $n \times 1$ real matrices and linear operators on $\mathbb{R}^n$ with  $n \times n$ matrices. A matrix $A \in S_n$ is said to be positive, $A \ge 0$, if $\langle Ax,x \rangle \ge 0$ for every $x \in \mathbb{R}^n$. This is equivalent to the fact that all the eigenvalues of $A$ are nonnegative. For $A,B \in S_n$ we write $A \le B$ if $B -A$ is positive. The partial order $\le$ is usually called Loewner's order. If $A \ge 0$ and $A$ is invertible, then $A$ is called a positive definite matrix. In this case we write $A > 0$. This happens if and only if all the eigenvalues of $A$ are positive. And finally, we write $A > B$ whenever $A-B$ is positive definite.

Let $\mathcal{M} ,\mathcal{N}  \subset S_n$. A map $\phi : \mathcal{M} \to S_n$ is called an order emebedding of $\mathcal{M}$ if for every pair $X,Y \in \mathcal{M}$ we have
$$
X \le Y \iff \phi (X) \le \phi (Y).
$$
A bijective order emebedding $\phi : \mathcal{M} \to \mathcal{N}$ is called an order isomorphism of $\mathcal{M}$ onto $\mathcal{N}$.

A subset $\U \subset S_n$ is called a matrix domain if it is open and connected. When speaking of domains we will always assume that they are nonempty.
Our goal is to describe the general form of embeddings of matrix domains. When dealing with order embeddings $\phi : \U \to S_n$ there is no loss of generality in assuming that $0 \in \U$ and $\phi (0) = 0$. Indeed, let $C \in \U$. Then the map $X \mapsto \phi (X + C) - \phi (C)$, $X \in \U-C = \{ Y-C\, : \, Y \in \U \}$ is an order embedding of the domain $\U - C$ which maps the zero matrix to itself. 

In order to formulate the main theorem of the paper we need to introduce a certian class of matrix domains. For every $A \in S_n$ we set
$$
\mathcal{W}_A = \{ X \in S_n \, : \, XA + I \ \, {\rm is} \ \, {\rm invertible} \}
$$
and denote by $\U_A$ the connected component of $\mathcal{W}_A$ that contains $0$.
Clearly, $\mathcal{W}_A$ is an open subset of $S_n$ and $\U_A$ is a matrix domain. Let us consider two extreme examples. If $A= 0$, then $\mathcal{W}_A = \U_A = S_n$. In the case when $A = I$ we have 
$$
\mathcal{W}_A = \mathcal{W}_I = \cup_{k=0}^n \mathcal{Z}_k,
$$
where $\mathcal{Z}_k$ denotes the set of all $n \times n$ symmetric matrices with $k$ eigenvalues smaller than $-1$ and $n-k$ eigenavalues larger than $-1$. Here, we have counted eigenvalues with their geometric multiplicity. It is easy to see that each of the subsets $\mathcal{Z}_k$, $k=0, \ldots, n$, is open and $\mathcal{Z}_0$ is connected. Thus, $\U_I = \mathcal{Z}_0$.

\begin{proposition}\label{predmain}
Let $A \in S_n$. The map $\psi_A : \mathcal{W}_A \to S_n$ given by
$$
\psi_A (X) = (XA + I)^{-1} X , \ \ \ X \in \mathcal{W}_A ,
$$
has the following properties:
\begin{itemize}
\item It is a bijection of $\mathcal{W}_A$ onto $\mathcal{W}_{-A}$ and its inverse is $\psi_{-A}$.
\item Both $\psi_A$ and $\psi_{-A}$ are continuous.
\item It is an order isomorphism of $\U_A$ onto $\U_{-A}$.
\end{itemize}
\end{proposition}

\begin{theorem}\label{main}
Let $n \ge 2$ be an integer and $\U \subset S_n$ a matrix domain such that $0 \in \U$. Assume that $\phi  : \U \to S_n$ is an order emebedding satisfying $\phi (0) =0$. Then there exist $A \in S_n$ and an invertible real $n \times n$ matrix $T$ such that $\U \subset \U_A$ and
\begin{equation}\label{form}
\phi (X) = T (XA + I)^{-1} XT^t
\end{equation}
for every $X \in \U$. Here, $T^t$ denotes the transpose of $T$.
\end{theorem}

We have the following straightforward consequence of the above two statements.

\begin{corollary}\label{konzul}
Let $n \ge 2$ be an integer and $\U \subset S_n$ a matrix domain. Assume that $\phi  : \U \to S_n$ is an order emebedding. Then $\phi (\U)$ is a matrix domain and $\phi : \U \to \phi (\U)$ is a homeomorphism and order isomorphism.
\end{corollary}

In the complex case the order embeddings $\phi : \U \to H_n$, where $H_n$ stands for the set of all $n \times n$ complex hermitian matrices and $\U \subset H_n$ is open and connected, were studied in \cite{MoS}. One of the main tools in \cite{MoS}
was infinite-dimensional holomorphy. Thus, the approach developed in \cite{MoS} does not work in the real case and it is not surprising that the main ideas in our paper are essentially different from those in \cite{MoS}. For example, a short and simple proof of the first item of Proposition \ref{predmain} is the same as in the complex case, the second item is trivial, while the proof of the only nontrivial item, whose complex analogue was based on the infinite-dimensional holomorphy, depends on a geometric idea that is inspired by  geometry of algebraic homogeneous spaces \cite{Cho} and  geometry of matrices \cite{Hu1, Hu2, Hu3, Hu4, Hu5, Hu6, Hu7, Hu8}. Let us just mention that some of our ideas are  inspired also by proof techniques in \cite{MoS2, Sem}. A more precise explanation will be given in the sequel. 

The next section will be devoted to preliminary results. The main theorem will be proved in the third section. The paper closes with final remarks. We will describe the general form of order embeddings of open matrix intervals, explain the optimality of our main result, give some further insights into obtained theorems that will reveal connections with other results such as the fundamental theorem of chronogeometry and finally discuss order embeddings of the closed matrix intervals showing in particular that they must be continuous everywhere but at the end points.

\section{Preliminary results}

For any $A,B \in S_n$ with $A \le B$ we define matrix intervals $(A,B)$ and $[A,B]$ by $$(A,B) = \{ X \in S_n \, : \, A < X < B \}$$ and
$$[A,B] = \{ X \in S_n \, : \, A \le X \le B \}.$$
Of course, $(A,B)$ is a nonempty set if and only if $A < B$.

Let us recall that for an arbitrary $A \in S_n$ and any positive real number $\varepsilon$ we have $D(A, \varepsilon) = \{ C \in S_n \, : \, \| C-A \| < \varepsilon \} = (A- \varepsilon I , A + \varepsilon I)$. Here, $\| \cdot \|$ stands for the usual operator norm.

The complex analogue of the lemma below was proved in \cite[Lemma 6.2]{MoS}. The same proof works also in the real case. As it is rather short we will repeat it here for the sake of completeness.

\begin{lemma}\label{automcont}
Let $\phi : (0,I) \to S_n$ be an order embedding. Let $\mathcal{C} \subset (0,1)$ be the set of all real numbers $t$, $0 < t < 1$, such that the map $\phi$ is continuous at the matrix $tI$. Then $(0,1) \setminus \mathcal{C}$ is at most countable.
\end{lemma}

\begin{proof}
Define a real function $f : (0,1) \to \mathbb{R}$ by
$$
f(t) = {\rm tr}\, \phi (tI), \ \ \ t \in (0,1).
$$
Here, ${\rm tr}\, A$ denotes the trace of $A$. Clearly, if $A,B \in S_n$ and $A \le B$ then ${\rm tr}\, A \le {\rm tr}\, B$. It follows that $f$ is an increasing function. Denote by  $\mathcal{C} \subset (0,1)$  the set of all real numbers $t$, $0 < t < 1$, such that the function $f$ is continuous at $t$. Then  $(0,1) \setminus \mathcal{C}$ is at most countable. 

Let $t_0 \in  \mathcal{C}$. We need to show that the map $\phi$ is continous at the matrix $t_0 I$. Choose any positive real number $\varepsilon$. Then we can find $\delta >0$ such that 
$t_0 + \delta < 1$, $t_0 - \delta > 0$, and
for every real $t$ we have
$| t - t_0 | \le \delta \Rightarrow | f(t) - f(t_0) | \le \varepsilon$. It follows that
$$
{\rm tr}\, (\phi ((t_0 + \delta) I) - \phi ((t_0 - \delta) I) ) $$ $$= {\rm tr}\, (\phi ((t_0 + \delta) I) - \phi (t_0 I) ) + {\rm tr}\, (\phi (t_0  I) - \phi ((t_0 - \delta) I) ) \le 2 \varepsilon.
$$
For every positive matrix $C$ we have $\| C \| \le {\rm tr}\, C$. Hence,
$$
\phi ((t_0 + \delta) I) \ge \phi ((t_0 - \delta) I)  
$$
yields
$$
\| \phi ((t_0 + \delta) I) - \phi ((t_0 - \delta) I) \|  \le 2 \varepsilon.
$$
It is well-known that for any $A,B,C,D \in S_n$ satisfying
$$
A \le C,D \le B \ \ \ {\rm and} \ \ \ \|B-A \| \le  2 \varepsilon
$$
we have $\| C-D \| \le 2\varepsilon$. For the sake of completeness we will verify this statement. We have $C - D \le B- D \le B-A$ and therefore, for every unit vector $x$ we have
$$
\langle (C-D) x,x \rangle \le \langle (B-A) x,x \rangle \le 2 \varepsilon.
$$
Similarly, we have $\langle (D-C) x,x \rangle \le  2 \varepsilon$, and consequently,  $| \langle (C-D) x,x \rangle | \le 2 \varepsilon$. This implies that $\max_{ \| x \| = 1} | \langle (C-D) x,x \rangle | =  \| C - D \| \le 2 \varepsilon$, as desired. 

Next, it is clear that if $A \in (0,I)$ and $\| A - t_0 I \| \le \delta$ then $(t_0 - \delta) I \le A \le (t_0 + \delta) I$. From here we conclude that
$$
\phi ((t_0 - \delta) I) \le \phi (t_0 I) , \phi (A) \le \phi ((t_0 + \delta) I). 
$$
Hence, we see that $A \in (0,I)$ and $\| A - t_0 I \| \le \delta$ yield $\| \phi (A) - \phi ( t_0 I) \| \le 2\varepsilon$. Thus, $\phi$ is continuous at $t_0 I$.
\end{proof}

We will need the following statement several times. The proof is trivial and will be omitted.

\begin{lemma}\label{trvtrv}
Let $A,B \in S_n$ be nonzero and $0 \le A \le B$. Set $p = {\rm rank}\, A$ and $q = {\rm rank}\, B$. Then 
the image of $A$ is contained in the image of $B$, and consequently, $p \le q$.
\end{lemma}

We denote by $P_{n}^1$ the set of all $n\times n$ projections of rank one.
Let $A,B \in S_n$ with $A \le B$. For every $P \in P_{n}^1$ we set
$$
\alpha (A,P;B) = \max \{ t \in \mathbb{R} \, : \, A + tP \le B\}
$$  
(note that the set $ \{ t \in \mathbb{R} \, : \, A + tP \le B\}$ is nonempty as it contains zero and is bounded above by $\| B -A \|$).
When identifying $n \times n$ matrices with linear operators on $\mathbb{R}^n$ we denote the image of $A$ by ${\rm Im}\, A$ and the kernel of $A$ by ${\rm Ker}\, A$. For $A \in S_n$ the restriction of $A$ to ${\rm Im}\, A$ is a bijective operator from ${\rm Im}\, A$ onto istelf. We denote by $A^\dagger : \mathbb{R}^n \to \mathbb{R}^n$ the Moore-Penrose inverse of $A$. It is defined in the following way. We know that $\mathbb{R}^n$ is the orthogonal direct sum of the image of $A$ and the kernel of $A$. Thus, any vector $x \in \mathbb{R}^n$ can be written as $x = x_1 + x_2$ for uniquely determined vectors $x_1 \in {\rm Im}\, A$ and $x_2 \in {\rm Ker}\, A$. Then
$$
A^\dagger x = \left( A_{| {\rm Im}\, A} \right)^{-1}  x_1 .
$$

 Lep $P$ be a rank one projection and $x$ a unit vector that spans ${\rm Im}\, P$. It is trivial to see that $P = xx^t$ (if we consider $P$ as a matrix, then the vector $x$ is represented by the $n \times 1$ matrix, and if $P$ is considered as a linear operator then $x^t$ denotes the linear functional on $\mathbb{R}^n$ given by $x^t z = \langle z,x \rangle$, $z \in  \mathbb{R}^n$).
The next lemma is just a slight modification of \cite[Theorem 4]{BG}.

\begin{lemma}\label{sformula}
Let $A,B \in S_n$ and $P = xx^t \in P_{n}^1$. Assume that $A \le B$. 
\begin{itemize}
\item If $x \in {\rm Im}\, (B-A)$ then 
\begin{equation}\label{uzuz}
\alpha (A,P;B) = { 1 \over {\rm tr}\, \left(    (B-A)^\dagger P \right) }=  { 1 \over   x^t (B-A)^\dagger x }.
\end{equation}
\item If $x \not\in {\rm Im}\, (B-A)$ then 
$$
\alpha (A,P;B) = 0.
$$
\end{itemize}
\end{lemma}

\begin{proof}
Assume first that the unit vector $x$ does not belong to the image of $B-A$. Then for every positive real number $t$ the image of $tP$ is not contained in the image of $B-A$, and therefore, $tP \not\le B-A$. It follows that $\alpha (A,P;B) = 0$. 

If, on the other hand, $x \in {\rm Im}\, (B-A)$, then with respect to the orthogonal direct sum decomposition $\mathbb{R}^n = {\rm Im}\, (B-A) \oplus  {\rm Ker}\, (B-A)$ the operators $P, B-A, (B-A)^\dagger$ have  matrix representations
$$
P = \left[ \begin{matrix} P_1 & 0 \cr 0 & 0 \cr \end{matrix} \right ] \, ,  \ \  B-A = \left[ \begin{matrix} (B-A)_1 & 0 \cr 0 & 0 \cr \end{matrix} \right ] \,  \ \ \ {\rm and} \ \ \ (B-A)^\dagger = \left[ \begin{matrix} (B-A)_{1}^{-1} & 0 \cr 0 & 0 \cr \end{matrix} \right ].
$$
We can restrict our attention to the upper left corners of these block matrices. In other words, we can assume that $B-A$ is invertible and $(B-A)^{-1} = (B-A)^\dagger$. 

The inequality
$$
(B-A) - tP \ge 0
$$
is equivalent to
$$
I - t(B-A)^{-1/2} P (B-A)^{-1/2} \ge 0.
$$
Because $t(B-A)^{-1/2} P (B-A)^{-1/2}$ is a rank one symmetric matrix whenever $t>0$, we have $t(B-A)^{-1/2} P (B-A)^{-1/2} \le I$ if and only if $${\rm tr}\, (t(B-A)^{-1/2} P (B-A)^{-1/2}) \le 1,$$ or equivalently,
$$
t \, {\rm tr}\, ((B-A)^{-1} P) \le 1.
$$
This yields (\ref{uzuz}).
\end{proof}

\begin{corollary}\label{picnar}
Assume that $A,B \in S_n$ satisfy $A < B$. Let $P\in P_{n}^1$ be any rank one projection and $t$ any real number.
Then $B - (A+tP)$ is invertible if $t \not= \alpha (A,P; B)$ and ${\rm rank}\, (B - (A+tP)) = n-1$  when $t = \alpha (A,P; B)$.
\end{corollary}

\begin{proof}
The rank of $B - (A+tP)$ is the same as the rank of $I - t(B-A)^{-1/2} P (B-A)^{-1/2}$. This matrix is invertible for all real $t$ but for $t$ for which the rank one matrix $t(B-A)^{-1/2} P (B-A)^{-1/2}$ is a rank one projection in which case this matrix is of rank $n-1$. Now, $t(B-A)^{-1/2} P (B-A)^{-1/2}$ is a rank one projection if and only if its trace is equal to $1$. By the above proof this happens exactly when $t = \alpha (A,P; B)$. 
\end{proof}

\begin{corollary}\label{trvdva}
Assume that $A,B \in S_n$ satisfy $0 < A < B$. Let $P\in P_{n}^1$.
Then
\begin{equation}\label{kkjjhh}
\alpha (0, P; A)  < \alpha (0, P; B).
\end{equation}
\end{corollary}
Note that (\ref{kkjjhh}) yields that  there exists a positive real number $s$ such that $\alpha (0, P; A) < s < \alpha (0, P; B)$ which further implies that
$$
sP \not\le A \ \ \ {\rm and} \ \ \ sP \le B.
$$

\begin{proof}
It is well-known that $0 < A < B$ yields $0 < B^{-1} < A^{-1}$ and therefore, Lemma \ref{sformula} implies
$$
\alpha (0, P; A) < \alpha (0, P; B).
$$
This completes the proof.
\end{proof}

\begin{corollary}\label{trvtriv}
Let $t,p,q$ be real numbers such that $0 < q < t < p $. Then there exist rank one projections $P,Q \in S_2$ such that
$$
t = \alpha ( 0, Q ; pP + q(I-P)) .
$$
\end{corollary}

\begin{proof}
Take any rank one projection $P \in S_2$. Because  $pP + q(I-P)$ is invertible in $S_2$ we have
\begin{equation}\label{porse}
 \alpha ( 0, Q ; pP + q(I-P)) = { 1 \over {\rm tr}\, \left( \left( {1 \over p} P + {1 \over q} (I-P) \right) Q \right) }
\end{equation}
for any rank one projection $Q \in S_2$. Clearly, $ \alpha ( 0, P ; pP + q(I-P)) = p$ and $ \alpha ( 0, I-P ; pP + q(I-P)) = q$. We complete the proof by observing that (\ref{porse}) depends continuously on $Q$ and the set of all $2 \times 2$ rank one projections is connected.  
\end{proof}

\begin{lemma}\label{less}
Let $A,B_1 , B_2 \in S_n$. Assume that $A \le B_j$, $j=1,2$. Then the following two statements are equivalent.
\begin{enumerate}
\item $B_1 \le B_2$.
\item For every rank one projection $P \in P_{n}^1$ we have $\alpha (A,P;B_1) \le \alpha (A,P;B_2)$.
\end{enumerate}
\end{lemma}

\begin{proof}
Let $A,B \in S_n$, $P \in P_{n}^1$, and assume that $A \le B$. Then clearly,
$$
\alpha (A,P;B) = \alpha (0,P;B-A).
$$
It follows that it is enough to prove the special case when $A=0$. This special case has been proved in \cite[Lemma 2.1]{Sem} using the ideas from \cite{BG}. 
\end{proof}

\begin{corollary}\label{equal1}
Let $A,B_1 , B_2 \in S_n$. Assume that $A \le B_j$, $j=1,2$. Then the following two statements are equivalent.
\begin{enumerate}
\item $B_1 = B_2$.
\item For every $P \in P_{n}^1$ we have $\alpha (A,P;B_1) = \alpha (A,P; B_2)$.
\end{enumerate}
\end{corollary}

\begin{lemma}\label{jztzavr}
Let $A,B \in S_n$ with $A < B$.
Let $(P_k) \subset P_{n}^1$ be a convergent sequence of projections of rank one. Denote $P = \lim P_k$. Then
$$
\lim \alpha (A,P_k;B) =  \alpha (A,P;B).
$$
\end{lemma}

\begin{proof}
The statement follows directly from
$$
\lim  {\rm tr}\, \left(    (B-A)^{-1} P_k \right) = {\rm tr}\, \left(    (B-A)^{-1} P \right) \not=0.
$$ 
\end{proof}

\begin{remark}
It would be tempting to conjecture that the above statement holds true if we replace the assumption $A < B$ by the weaker assumption that $A \le B$. An easy example shows that this conjecture is wrong. Take any convergent sequence of rank one projections $(P_k)$ satisfying $P_k \not= P = \lim P_k$, $k=1,2,\ldots$ Then for every positive integer $k$ we have $\alpha (0,P_k;P) =0$ and therefore
$$
\lim \alpha (0,P_k;P) = 0,
$$
while $ \alpha (0,P;P) = 1$. A slight modification of this example shows that when $A \le B$ but $A \not< B$ we can find a  convergent sequence of rank one projections $(P_k)$ such that the sequence $(\alpha (A,P_k;B) )$ is not convergent.
\end{remark}

\begin{corollary}\label{equal}
Let $A,B_1 , B_2 \in S_n$, $A < B_j$, $j=1,2$, and $\mathcal{P} \subset P_{n}^1$ be a dense subset. Then the following two statements are equivalent.
\begin{enumerate}
\item $B_1 = B_2$.
\item For every rank one projection $P \in \mathcal{P}$ we have $\alpha (A,P;B_1) = \alpha (A,P; B_2)$.
\end{enumerate}
\end{corollary}

\begin{proof}
The nontrivial direction is a straightforward consequence of Corollary \ref{equal1} and Lemma \ref{jztzavr}.
\end{proof}

\begin{lemma}\label{jztzb}
Let $A,B \in S_n$ with $A < B$.
Let $(B_k) \subset S_n$ be a convergent sequence such that $B = \lim B_k$ and $A \le B_k$, $k=1,2,\ldots$ Then for every $P \in P_{n}^1$ we have
$$
\lim \alpha (A,P;B_k) =  \alpha (A,P;B).
$$
\end{lemma}

\begin{proof}
It follows from $A < B$ and $B = \lim B_k$ that there exists a positive integer $k_0$ such that $A < B_k$ for every integer $k > k_0$. For all such integers $k$ we have
$$
\alpha (A,P;B_k)    = { 1 \over {\rm tr}\, \left(    (B_k-A)^{-1} P \right) }.
$$
The desired conclusion is a straightforward consequence. 
\end{proof}

\begin{remark}
Again, the statement is not true under the weaker assumption that $A \le B$ even if we assume that $A < B_k$ for every positive integer $k$. 
Indeed, take $A = 0$ and let $(P_k)$ be a convergent sequence of rank one projections such that $P_k \not= P = \lim P_k$, $k=1,2,\ldots$  Note that for any pair of rank one projections $Q,R$ we have $0 \le {\rm tr}\, (QR) \le 1$, and ${\rm tr}\, (QR) = 1$ if and only if $Q=R$. Thus,  for every positive integer $k$ we have $1 \ge {\rm tr}\, (P (I-P_k)) = 1 - {\rm tr}\, (P P_k) > 0$. Define a sequence $(B_k) \subset S_n$ by
$$
B_k = P_k + c_k ( I - P_k) , \ \ \ k=1,2,3,\ldots,
$$
where
$$
0 < c_k = { {\rm tr}\, (P (I-P_k)) \over k} \le {1 \over k}, \ \ \ k=1,2,3,\ldots 
$$
Therefore, $B_ k > 0$, $k=1,2,\ldots$, and $B = \lim B_k = P$. Further, for every positive integer $k$ we can calculate
$$
\alpha (0,P;B_k) =  { 1 \over {\rm tr}\, \left(    B_{k}^{-1} P \right) }  =  { 1 \over {\rm tr}\, \left(    (  P_k + c_k ( I - P_k))^{-1} P \right) }
$$
$$
=  { 1 \over {\rm tr}\, \left(    (  P_k + c_{k}^{-1} ( I - P_k)) P \right) } =  { 1 \over {\rm tr}\, (P_k P)  + k} \le {1 \over k},
$$
and consequently,
$$
\lim \alpha (0,P;B_k) = 0,
$$
while $ \alpha (0,P;B) = 1$.
\end{remark}

\begin{lemma}\label{nicce}
Let $n \ge 2$ be an integer and $\phi : [0,I] \to S_n$ an order embedding satisfying $\phi (0) = 0$. Assume that $\phi$ is continuous at $0$. Then $\phi$ is injective and ${\rm rank}\, \phi (X) = {\rm rank}\, X$ for every $X \in [0,I]$.
\end{lemma}

\begin{proof}
Assume that $\phi (X) = \phi (Y)$ for some $X,Y \in  [0,I]$. Then $\phi (X) \le \phi (Y)$ and therefore $X \le Y$, and also $Y \le X$. It follows that $X=Y$. 

We will first show that for any pair $X,Y \in  [0,I]$ satisfying $X \le Y$ and ${\rm rank}\, X < {\rm rank}\, Y$ we have ${\rm rank}\, \phi (X) < {\rm rank}\, \phi(Y)$. We first observe that $X \le Y$ implies $0 \le \phi(X) \le \phi(Y)$, and therefore,
${\rm rank}\, \phi (X) \le {\rm rank}\, \phi(Y)$. We have to show that ${\rm rank}\, \phi (X) \not= {\rm rank}\, \phi(Y)$. Assume on the contrary that ${\rm rank}\, \phi (X) = {\rm rank}\, \phi(Y)$. Because  ${\rm Im}\, \phi (X) \subset {\rm Im}\, \phi(Y)$ this yields 
${\rm Im}\, \phi (X) = {\rm Im}\, \phi(Y)$.  Denote by $s$ the smallest nonzero eigenvalue of $\phi (X)$.

Because ${\rm rank}\, X < {\rm rank}\, Y$ we have
\begin{equation}\label{aaq}
tY \not\le X
\end{equation}
for every positive real number $t$. The continuity of $\phi$ at $0$ implies the existence of a positive real number $t_0 \le 1$ such that for every positive real $t$, $t< t_0$, we have $0 \le \phi (tY) \le sI$. This together with ${\rm Im}\, \phi (tY) \subset {\rm Im}\, \phi (Y) = {\rm Im}\, \phi (X)$ yields $\phi (tY) \le \phi (X)$, contradicting (\ref{aaq}).

In the next step we will show that ${\rm rank}\, \phi (X) \le {\rm rank}\, X$ for every $X \in [0, I]$. If ${\rm rank}\, X = k$ then obviously we can find projections $P_{k+1}, \ldots, P_n = I$ such that 
$X \le P_{k+1} \le \ldots \le P_{n-1} \le I$ and ${\rm rank}\, P_r = r$, $r = k+1, \ldots, n$. Using the above claim together with ${\rm rank}\, \phi (I) \le n$ we arrive at the desired inequality 
${\rm rank}\, \phi (X) \le {\rm rank}\, X$.

It remains to prove that  ${\rm rank}\, \phi (X) \ge {\rm rank}\, X$ for every nonzero $X \in [0, I]$. We denote ${\rm rank}\, X = k$. Using the spectral decomposition of $X$ we can easily find $X_1 , \ldots X_k = X \in [0, I]$
such that $X_1 \le \ldots \le X_k$ and ${\rm rank}\, X_r = r$, $r=1, \ldots, k$. Because $\phi (X_1)$ is nonzero the above claim implies that  
${\rm rank}\, \phi (X) \ge {\rm rank}\, X$.
\end{proof}

\begin{corollary}\label{posle1}
Let $n \ge 2$ be an integer and $\phi : [0,I] \to S_n$ an order embedding satisfying $\phi (0) = 0$ and $\phi (I) = I$. Assume that $\phi$ is continuous at $0$ and $I$. Then  for every $X \in [0,I]$ of rank one we have $$\phi (I-X) = I -R$$
for some rank one matrix $R \in [0,I]$.
\end{corollary}

\begin{proof}
Define a map $\psi : [0,I] \to S_n$ by $\psi (X) = I - \phi (I-X)$, $X \in [0,I]$. Clearly, $\psi$ is an order embedding that is continuous at $0$ and $I$ and $\psi (0) = 0$ and $\psi (I) = I$. The desired conclusion is now a straightforward consequence of Lemma \ref{nicce}.
\end{proof}

\begin{corollary}\label{posl2}
Let $n \ge 2$ be an integer and $\phi : [0,I] \to S_n$ an order embedding satisfying $\phi (0) = 0$ and $\phi (I) = I$. Assume that $\phi$ is continuous at $0$ and $I$. Then there exists an injective map $\varphi : P_{n}^1 \to P_{n}^1$ and for every $P \in P_{n}^1$ there exists a strictly increasing function $f_P : [0,1] \to [0,1]$ such that 
$$\phi (P) = \varphi (P), \ \ \ P \in P_{n}^1,$$ and  $$\phi (tP) = f_P (t) \varphi (P), \ \ \ P \in P_{n}^1, \ \, t \in  [0,1].$$ 
\end{corollary}

\begin{proof}
Let $P \in P_{n}^1$. Then we already know that $\phi (P)$ is of rank one and $0 \le \phi (P) \le I$. Hence, $\phi (P) = sR$ for some rank one projection $R$ and some real number $s$, $0 \le s \le 1$. We need to show that $s=1$. Assume on the contrary that $s < 1$. Since $\phi$ is continuous at $I$ and $\phi (I) = I$ there exists a  real number $\delta > 0$ such that $A \in   [0,I]$ and $\| A - I \| < \delta$ yield $\phi (A) > sI$. Here, $\| \cdot \|$ stands for the usual operator norm. In particular, 
$$
\phi \left( I - {\delta \over 2} I \right) > sI \ge sR = \phi (P),
$$
and consequently, $ I - {\delta \over 2} I \ge P$, a contradiction. Therefore, there exists a map $\varphi : P_{n}^1 \to P_{n}^1$ such that $\phi (P) = \varphi (P)$  for every  $P \in P_{n}^1$.
It is injective by Lemma \ref{nicce}.

Let $r$ be a real number, $0 \le r \le 1$.
It follows from $0 \le rP \le P$ that $\phi (rP) = t \varphi(P)$ for some real number $t$, $0 \le t \le 1$. The function $f_P : [0,1] \to [0,1]$ which sends $r$ to $t$ is obviously strictly increasing.
\end{proof}

\begin{lemma}\label{kdajboze}
Let $n \ge 2$ be an integer and $\phi : [0,I] \to S_n$ an order embedding satisfying $\phi (0) = 0$ and $\phi (I) = I$. Assume that $\phi$ is continuous at $0$ and $I$.
 Let $\varphi : P_{n}^1 \to P_{n}^1$ 
and $f_P : [0,1] \to [0,1]$, $P \in P_{n}^1$, be as in Corollary \ref{posl2}. 

Assume that $A \in (0,I)$, $P\in P_{n}^1$, and that $(A_k) \subset (0,I)$ is a sequence such that $A < A_k$, $k =1,2,\ldots$, and $\lim A_k = A$, and
$\lim \phi (A_k) = \phi (A)$. Denote $t_0 = \alpha (0, P; A)$. Then there exists a decreasing sequence of real numbers $(t_k ) \subset (t_0 ,1)$ (note that then the sequence $(f_P (t_k))$ is decreasing and bounded below, and therefore convergent) such that
$$
\lim t_k = t_0
$$
and
$$
f_P (t_0) \le lim f_P (t_k) \le \alpha (0, \varphi (P); \phi (A)).
$$
\end{lemma}

\begin{proof}
By the remark following Corollary \ref{trvdva} we can find for every positive integer $k$ a real number $t_k$ such that $t_k P \not\le A$ and $t_k P \le A_k$. Clearly, each $t_k$ belongs to the open interval $(t_0 ,1)$. It is easy to modify the sequence $(t_k)$ to be a decreasing sequence without spoiling the properties $t_k P \not\le A$, $t_k P \le A_k$, $t_k \in (t_0,1)$, $k=1,2,\ldots$  By Lemma \ref{jztzb},
$$ 
\lim \alpha (0,P; A_k) = t_0.
$$
Combining this with $t_k \le \alpha (0, P; A_k)$ we see that $\lim t_k = t_0$.

For every positive integer $k$ we know that
$$
\phi (A_k)\ge  f_P (t_k) \varphi (P) .
$$
Sending $k$ to infinity we arrive at $\phi (A) \ge \lim f_P (t_k) \varphi (P)$ which finally gives
$f_P (t_0) \le lim f_P (t_k) \le \alpha (0, \varphi (P); \phi (A))$.
\end{proof}

\begin{lemma}\label{kdajbize}
Let $n \ge 2$ be an integer and $\phi : [0,I] \to S_n$ an order embedding satisfying $\phi (0) = 0$ and $\phi (I) = I$. Assume that $\phi$ is continuous at $0$ and $I$.
 Let $\varphi : P_{n}^1 \to P_{n}^1$ 
and $f_P : [0,1] \to [0,1]$, $P \in P_{n}^1$, be as in Corollary \ref{posl2}. 

Assume that $A \in (0,I)$, $P\in P_{n}^1$, and that $(A_k) \subset (0,I)$ is a sequence such that $A > A_k$, $k =1,2,\ldots$, and $\lim A_k = A$, and
$\lim \phi (A_k) = \phi (A)$. Denote $t_0 = \alpha (0, P; A)$. Then there exists an increasing sequence of real numbers $(t_k ) \subset (0, t_0)$ (note that then the sequence $(f_P (t_k))$ is increasing and bounded above, and therefore convergent) such that
$$
\lim t_k = t_0
$$
and
$$
f_P (t_0) \ge lim f_P (t_k) \ge \alpha (0, \varphi (P); \phi (A)).
$$
\end{lemma}

\begin{proof}
By Corollary \ref{trvdva} we can find for every positive integer $k$ a positive real number $t_k$ such that $t_k  < t_0$ and $t_k P \not\le A_k$.  We can further assume that $(t_k)$ is an increasing sequence.  By Lemma \ref{jztzb},
$
\lim \alpha (0,P; A_k) = t_0$.
This together with $t_k > \alpha (0, P; A_k)$ yields $\lim t_k = t_0$.
For every positive integer $k$ we have 
$f_P (t_0) >  f_P (t_k) > \alpha (0, \varphi (P); \phi (A_k))$. It follows from Lemma \ref{nicce} that $\phi (A) > 0$. We complete the proof by taking the limits as $k$ tends to infinity.
\end{proof}

We have a simple consequence of the previous two lemmas.

\begin{corollary}\label{kdajbuze}
Let $n \ge 2$ be an integer and $\phi : [0,I] \to S_n$ an order embedding satisfying $\phi (0) = 0$ and $\phi (I) = I$. Assume that $\phi$ is continuous at $0$ and $I$.
 Let $\varphi : P_{n}^1 \to P_{n}^1$ 
and $f_P : [0,1] \to [0,1]$, $P \in P_{n}^1$, be as in Corollary \ref{posl2}. 

Assume that $A \in (0,I)$ and $P\in P_{n}^1$.
Suppose that there are two sequences $(A_k), (B_k) \subset (0,I)$ such that $A_k < A < B_k$, $k =1,2,\ldots$, $\lim A_k = A = \lim B_k$, and
$\lim \phi (A_k) = \phi (A) = \lim \phi (B_k)$. Denote $t_0 = \alpha (0, P; A)$. Then $f_P$ is continuous at $t_0$ and
$$
f_P (t_0) = \alpha (0, \varphi (P); \phi (A)).
$$
\end{corollary}

Projections $P,Q \in P_{n}^1$ are said to be orthogonal, $P \perp Q$, if $PQ=0$.

\begin{corollary}\label{servkia}
Let $n \ge 2$ be an integer and $\phi : [0,I] \to S_n$ an order embedding satisfying $\phi (0) = 0$, $\phi ((1/2)I) = (1/2)I$, and $\phi (I) = I$. Assume that $\phi$ is continuous at $0$, $(1/2)I$, and $I$. Let $\varphi : P_{n}^1 \to P_{n}^1$ 
and $f_P : [0,1] \to [0,1]$, $P \in P_{n}^1$, be as in Corollary \ref{posl2}.
Then for every pair $P,Q \in P_{n}^1$ we have
\begin{itemize}
\item $\phi ((1/2)P) = (1/2)\varphi (P)$,
\item $P \perp Q \Rightarrow \phi (P) \perp \phi (Q)$, 
\item $P \perp Q \Rightarrow \phi (P + Q) = \varphi(P) + \varphi (Q)$, and
\item $P \perp Q \Rightarrow \phi ((1/2)P + Q) = (1/2)\varphi(P) + \varphi (Q)$.
\end{itemize}
\end{corollary}

\begin{proof}

To prove the first item we need to show that
$f_P ((1/2)) = 1/2$. To verify this we apply Corollary \ref{kdajbuze} with $A = (1/2)I$, $A_k = (1/2)I - {1 \over k+2} I$, and $B_k = (1/2)I + {1 \over k+2} I$, $k=1,2,\ldots$ Moreover, $f_P$ is continuous at $1/2$.

Let now $P,Q$ be a pair of orthogonal projections of rank one. We know that $\phi (P + Q)$ is of rank two. Moreover, $\varphi(P), \varphi (Q) \le \phi (P +Q) \le I$. Since  $\varphi(P)$ and $\varphi (Q)$ are projections of rank one satisfying $\varphi (P) \not= \varphi (Q)$ one can easily verify that $\phi (P+Q)$ is a projection of rank two.

In the same way we see that there exists a projection $R$ of rank two such that $\phi ( (1/2) (P+Q)) = (1/2)R$. Because $\phi ( (1/2) (P+Q)) \le \phi (P+Q)$ we conclude that $\phi ( (1/2) (P+Q)) = (1/2)\phi (P+Q)$. And then, clearly
$$
(1/2)\phi (P+Q) \le \phi ((1/2)P +Q) \le \phi (P+Q).
$$
It follows that there exist real numbers $p,q$, $1/2 \le p,q \le 1$, and a pair of orthogonal rank one projections $R_1 , R_2$  such that $R_1 + R_2 = \phi (P+Q)$ and
$$
 \phi ((1/2)P +Q) = pR_1 + qR_2.
$$
Since $\varphi (Q) \le pR_1 + qR_2 \not=\phi (P+Q)$ we see that  one of the real numbers $p,q$ is equal to one. With no loss of generality we may assume that $q=1$
and then
$$
\varphi (Q) \le pR_1 + R_2
$$
and $p < 1$ yield that $R_2 = \varphi (Q)$.
It is an elementary linear algebra exercise to show that already known facts $1/2 \le p$, $f_P (1/2) = 1/2$, the continuity of $f_P$ at $1/2$, and
 $$
f_P (t) \varphi (P) \not\le pR_1 + \varphi (Q), \ \ \ t\in (1/2, 1],
$$
imply that $R_1 = \varphi (P)$ and $p = 1/2$. In particular, $\varphi(P) \perp \varphi(Q)$.

Since $\varphi(P) ,\varphi(Q) \le \phi (P+Q)$ and $\phi (P+Q)$ is a projection of rank two we conclude that  $\phi (P + Q) = \varphi(P) + \varphi (Q)$.
\end{proof}

We will need the following easy technical result \cite[Lemma 2.4]{Sem}. Actually, it is an almost straightforward consequence of Corollary \ref{trvtriv}.

\begin{lemma}\label{mikmik}
Let $n \ge 2$ be an integer, $R \in P_{n}^1$ any projection of rank one, and $s$ a real number, $1/2 < s < 1$. Then there exists a pair of orthogonal rank one projections $P,Q \in P_{n}^1$ such that for every real number $p$ we have
$$
pR \le (1/2) P + Q \iff p \le s.
$$
\end{lemma}

\begin{proposition}\label{nice}
Let $n \ge 3$ be an integer and $\phi : [0,I] \to S_n$ be an order embedding satisfying $\phi (0) = 0$, $\phi ((1/2)I) = (1/2)I$, and $\phi (I) = I$. Assume that $\phi$ is continuous at $0$, $(1/2)I$, and $I$. Then there exists an orthogonal $n \times n$ matrix $O$ such that
$\phi (X) = OXO^t$ for every $X \in [0,I]$.
\end{proposition}

\begin{proof}
We know that there exists an orthogonality preserving injective map $\varphi : P_{n}^1 \to P_{n}^1$ such that for every $P \in P_{n}^1$ we have
$\phi (P) = \varphi (P)$. We claim that there exists an $n \times n$ orthogonal matrix $O$ such that $\varphi (P) = OPO^t$ for every $P \in  P_{n}^1$.
This follows directly from \cite[Theorem 4.1]{Fa} once we verify that for every triple $P,Q,R \in P_{n}^1$ we have
$$
{\rm Im}\, P \subset {\rm Im}\, R + {\rm Im}\, Q \Rightarrow {\rm Im}\, \varphi (P) \subset {\rm Im}\, \varphi (R) + {\rm Im}\, \varphi (Q).
$$
The verification of this implication is trivial in the special case when $R=Q$. Thus, assume that $R \not=Q$. Then there exist $n-2$ pairwise orthogonal rank one projections $P_1 , \ldots , P_{n-2}$ such that
$$
R,Q \perp P_j , \ \ \ j= 1, \ldots , n-2,
$$ 
and consequently, $P \perp P_j$,  $j= 1, \ldots , n-2$. It follows that
$$
\varphi (P), \varphi (R), \varphi (Q) \perp \varphi (P_j ), \ \ \ j= 1, \ldots , n-2,
$$ 
and we also know that $\varphi (P_i) \perp \varphi (P_j)$ whenever $i \not= j$. We further know that $\varphi$ is injective and therefore $\varphi (R) \not=\varphi (Q)$. It follows that
$$
{\rm Im}\, \varphi (P) \subset {\rm Im}\, \varphi (R) + {\rm Im}\, \varphi (Q),
$$
as desired.

After replacing the map $\phi$ by the map $X \mapsto O^t \phi(X) O$ we can assume with no loss of generality that $\phi (P) = P$ for every $P \in P_{n}^1$. It follows from Corollary \ref{servkia} that
$\phi ((1/2) P + Q) = (1/2)P+Q$ for every pair $P,Q \in P_{n}^1$ satisfying $P \perp Q$.

The proof will be completed by using exactly the same arguments as in the proof of \cite[Theorem 3.1]{Sem}. As the rest of the proof is rather short we will include it here for the sake of completeness.

In the next step we verify that
\begin{equation}\label{oppos}
\phi (I-P) = I-P
\end{equation}
 for every $P \in P_{n}^1$.
To show this we notice that every rank one projection $Q$ that is orthogonal to $P$ satisfies $Q \le I -P$ and consequently, $Q \le \phi (I-P)$. From here we deduce that $I-P \le \phi (I-P)$. It follows that $\phi (I-P) = (I-P) + tP$ for some $t \in [0,1]$. For any real $s \in (0,1]$ we have $sP \not\le I-P$. Thus, for every real $p \in (0,1]$ we have $pP \not\le \phi (I-P)$ (here, we have used the continuity of $\phi$ at $0$) which yields  that $t=0$.

Lemma \ref{mikmik} tells that  for every $R \in P_{n}^1$ and every real $t \in [1/2,1]$ 
we have $\phi (tR) = tR$.

We define a new map $\eta : [0,I] \to [0,I]$ by $\eta (X) = I - \phi (I-X)$, $X \in [0,I]$. Obviously, $\eta$ is an order embedding of $[0,I]$ satisfying $\eta(qI) = qI$, $q=0,1/2,1$. From (\ref{oppos}) we infer that 
$\eta (P) = P$ for every $P \in P_{n}^1$. But then we already know that $\eta (tR) = tR$ for all real $t \in [1/2,1]$ and $R \in P_{n}^1$. Hence, $\phi (I -tR) = I -tR$ 
 for every $R \in P_{n}^1$ and every real $t \in [1/2,1]$. 

 We know that there exists a strictly increasing function $f : [0 , 1] \to [0 , 1] $ (depending on $R$) such that $\phi (tR) = f(t) R$ for every $R \in P_{n}^1$ and every real $t \in [0,1]$.

Let $s \in [0, 1/2]$ and take any $R \in P_{n}^1$.
Then $t = 1-s \in [1/2 , 1]$. For every real number $p \in [0,1]$ the inequality $p \le s$ is equivalent to $pR \le I -tR$ which happens if and only if $f (p) R \le I - tR$. Hence, for 
every real $p \in [0,1]$ we have $p \le s$ if and only if $f(p) \le 1-t =s$. Thus, $f (s) = s$.

We have shown that $\phi (tR) = tR$ for every real $t \in [0,1]$ and every $R \in P_{n}^1$. 
The proof is completed by a direct use of Corollary \ref{equal1}.
\end{proof}

\begin{proposition}\label{nice2}
Let $0$ and $I$ stand for $2\times 2$ zero matrix and identity matrix, respectively, and let  $\phi : [0,I] \to S_2$ be an order embedding satisfying $\phi (0) = 0$, $\phi ((1/2)I) = (1/2)I$, and $\phi (I) = I$. Assume that $\phi$ is continuous at $0$, $(1/2)I$, and $I$. Then there exists an orthogonal $2 \times 2$ matrix $O$ such that
$\phi (X) = OXO^t$ for every $X \in [0,I]$.
\end{proposition}

\begin{proof}
By Corollaries \ref{posl2} and \ref{servkia}  
there exists an injective map $\varphi : P_{2}^1 \to P_{2}^1$ such that for every $P\in P_{2}^1$ we have
$\phi (P) = \varphi (P)$, $\phi ((1/2)P) = (1/2)\varphi (P)$, $\phi (I-P) = I - \phi (P)$,  and  $\phi ((1/2)P + (I-P)) = (1/2)\varphi(P) + (I-\varphi (P))$.
Moreover,  for each $P\in P_{2}^1$ there exists a strictly increasing function $f_P : [0,1] \to [0,1]$ such that $\phi (tP) = f_P (t) \varphi (P)$, $t \in  [0,1]$.
We will often use the notation $P^\perp = I-P$.

Let $\mathcal{C} \subset  [0,1]$ be the set of all points $s \in  [0,1]$ such that the mapping $t \mapsto \phi (tI)$, $t \in  [0,1]$, is continuous at $s$. Corollary \ref{kdajbuze} implies that for all $t \in \mathcal{C} \cap (0,1)$ and $P\in P_{n}^1$ the function $f_P$ is continuous at $t$ and
\begin{equation}\label{mraknaoc22}
f_P (t) = \alpha (0, \varphi (P); \phi (tI) ).
\end{equation}
It is trivial to see that $f_P (0) = 0 = \alpha (0, \varphi (P); 0 ) = \alpha (0, \varphi (P); \phi (0) )$ and that $f_P$ is continuous at $0$. And similarly, 
$f_P (1) = 1 = \alpha (0, \varphi (P); I ) = \alpha (0, \varphi (P); \phi (I) )$ and $f_P$ is continuous at $1$.

Let $P$ be any projection of rank one and $t \in (0,1)$. From $P \le P + tP^\perp$ we get $\varphi (P) = \phi (P) \le \phi (P + tP^\perp) \le I$ which together with $\phi (P + tP^\perp) \not=I, \varphi(P)$ yields that $\phi ( P + tP^\perp ) = \varphi (P) + s \varphi(P)^\perp$ for some $s \in (0,1)$.
If we additionaly assume that $t \in \mathcal{C}$ then the fact that for every $r \in (0,1)$ we have $rP^\perp \le P + tP^\perp \iff r \le t$ implies that $\phi ( P + tP^\perp ) = \varphi (P) + f_{P^\perp} (t) \varphi(P)^\perp$. 

Let $p,q \in \mathcal{C} \setminus \{ 0,1 \}$ and $P \in P_{n}^1$. Then $qP^\perp  \le     pP + q P^\perp \le P + q P^\perp$, and therefore, $f_{P^\perp} (q) \varphi(P)^\perp \le \phi ( pP + q P^\perp) \le \phi (P + q P^\perp) =
\varphi (P) +  f_{P^\perp} (q) \varphi(P)^\perp$, which further yields that $\phi ( pP + q P^\perp  ) = r\varphi (P) +  f_{P^\perp} (q) \varphi(P)^\perp$ for some $r \in (0,1)$. Following the same idea as in the previous paragraph we finally arrive at  $\phi ( pP + q P^\perp ) = f_P (p)\varphi (P) +  f_{P^\perp} (q) \varphi(P)^\perp$.

Let $p$ be any real number from the closed unit interval belonging to $\mathcal{C}$ and $P,Q$ any pair of projections of rank one such that $\varphi (P) \not= \varphi (Q)$ and $\varphi (P) \not= \varphi (Q)^\perp$ (such a pair exists because $\varphi$ is injective).
Then we know that
$$
\phi (pI)= f_P (p)\varphi (P) +  f_{P^\perp} (p) \varphi(P)^\perp
= f_Q (p)\varphi (Q) +  f_{Q^\perp} (p) \varphi(Q)^\perp .
$$
Assume that $f_P (p) \not= f_{P^\perp} (p)$. Then real numbers $f_P (p)$ and $f_{P^\perp} (p)$ are the eigenvalues of $\phi (pI)$ and the images of rank one projections $\varphi (P)$ and $\varphi(P)^\perp$ are the corresponding one-dimensional eigenspaces. Of course, we have $\{ f_P (p) , f_{P^\perp} (p) \} = \{ f_Q (p) , f_{Q^\perp} (p) \}$ and the images of rank one projection $\varphi (Q)$ and $\varphi(Q)^\perp$ are the two one-dimensional eigenspaces of $\phi (pI)$. Consequently, $\varphi (P) = \varphi (Q)$ or $\varphi (P) = \varphi (Q)^\perp$, a contradiction.

It follows that
$$
\phi (pI) = f(p) I , \ \ \ p \in \mathcal{C},
$$
for some  strictly increasing function $f : [0,1] \to [0,1]$, and
$$
\phi (pP + qP^\perp) = f(p) \varphi (P) + f(q) \varphi (P)^\perp
$$
for all $ p,q \in \mathcal{C}$ and every $P \in P_{n}^1$.

Recall that $ [0,1] \setminus \mathcal{C}$ is at most countable. For every $t \in (0,1)$ we can use Corollary \ref{trvtriv} to find $p,q \in (0,1)$ and rank one projections $P,Q$ such  that $q  < t < p$, $p,q \in \mathcal{C}$, and 
$$
t = \alpha (0,Q ; pP + qP^\perp) .
$$
By Corollary \ref{kdajbuze} the function $f_Q = f$ is continuous at $t$. It follows that $f$ is a continuous bijection of $[0,1]$ onto itself and we have
$$
\phi (pP + qP^\perp) = f(p) \varphi (P) + f(q) \varphi (P)^\perp
$$
for all $ p,q \in [0,1]$ and every $P \in P_{n}^1$. In the next step we will prove that $\varphi$ is surjective. This will together with the above formula imply that $\phi$ is a bijection of $[0,I]$ onto itself. 
The result is then a straightforward consequence of the main theorem in \cite{Sem}.

In order to prove the surjectivity of $\varphi$ we only need to verify that $\varphi : P_{2}^1 \to P_{2}^1$ is continuous. The surjectivity then follows from the injectivity of $\varphi$ and the well-known fact that the projective space $P_{2}^1$ is homeomorphic to the circle. 

So, assume that $(P_k) \subset P_{2}^1$ is a convergent sequence with $\lim P_k = P$. The compactness of the circle yields the existence of a cluster point $Q$ of the sequence  $(\varphi (P_k))$. All we need to show
is that this cluster point is unique (and therefore, it is a limit of this sequence) and $Q = \varphi (P)$. 

By passing to a subsequence we may assume that $Q = \lim \varphi (P_k)$.
Clearly,
$$
1 = \alpha (0,P; P + (1/2) P^\perp).
$$
Using $f(1) =1$, the continuity of $f$, and Lemma \ref{jztzavr} we see that
$$
1 = f( \alpha (0,P; P + (1/2) P^\perp)) = \lim_{k \to \infty} f( \alpha (0,P_k; P + (1/2) P^\perp)).
$$
It follows then from Lemma \ref{jztzb} that
$$
1 =  \lim_{k \to \infty} \left( \lim_{m \to \infty} f( \alpha (0,P_k; (1 - (1/m))P + (1/2) P^\perp)) \right).
$$
By Corollary \ref{kdajbuze} we have
$$
1 =  \lim_{k \to \infty} \left( \lim_{m \to \infty}  \alpha (0, \varphi (P_k); \phi((1 - (1/m))P + (1/2) P^\perp)) \right) .
$$
From here we deduce that
$$
1 =  \lim_{k \to \infty} \left( \lim_{m \to \infty}  \alpha (0, \varphi (P_k); f(1 - (1/m))\varphi(P) + (1/2) \varphi (P)^\perp) \right)
$$
$$
 =  \lim_{k \to \infty}  \alpha (0, \varphi (P_k); \varphi(P) + (1/2) \varphi (P)^\perp) = \alpha (0, Q; \varphi(P) + (1/2) \varphi (P)^\perp).
$$
It follows that $Q = \varphi (P)$, as desired.
\end{proof}

Our next goal is to prove Proposition \ref{predmain}. We will repeadetly use the well-known fact that if a matrix $BA + I$ is invertible for some square matrices $A,B$ then $AB + I$ is invertible, too. Indeed, one can easily verify that $I- A(BA + I)^{-1} B$ is the inverse of $AB + I$. We will also need the following well-known easy fact. Assume that $A,B \in S_n$ and $A \le B$. Let $j,k$ be integers, $1 \le j,k \le n$. If $A$ has $j$ positive eigenvalues then $B$ has at least $j$ positive eigenvalues.  Here, of course, we count eigenvalues with their mulitplicities. If $B$ has $k$ negative eigenvalues then $A$ has at least  $k$ negative eigenvalues.

We start with a trivial lemma.

\begin{lemma}\label{trivialn}
Let $n \ge 2$, $\U \subset S_n$ be a matrix domain, and $Z \in \U$. Denote by $\mathcal{W}$ the set of all elments $X \in \U$ with the property that we can find a positive integer $m$ and matrices $A_1 , \ldots , A_m \in \U$ and $B_1 , \ldots B_m \in \U$ such that
\begin{itemize}
\item $A_j < B_j$ and $[A_j , B_j] \subset \U$, $j=1, \ldots , m$, and
\item the intersection $ [A_j , B_j] \cap  [A_{j+1} , B_{j+1}]$ has a nonempty interior, $j=1, \ldots, m-1$, and
\item $Z \in  (A_1 , B_1)$ and $X \in  (A_m , B_m)$.
\end{itemize}
Then $\U = \mathcal{W}$.
\end{lemma}

\begin{proof}
All we need to do is to verify the trivial statement that $\mathcal{W}$ is nonempty, and both open and closed.
\end{proof}

The following lemma is just a minor modification of an analogous statement in the complex case \cite{MoS}.

\begin{lemma}\label{uzivee}
Let $n \ge 2$ be an integer and $A\in S_n$. Let $\U_A$ be defined as in Introduction. If $X,Y \in \U_A$ and $X < Y$, then $[X, Y] \subset \U_A$. 
\end{lemma}

\begin{proof}
By spectral theorem we know that there exist an integer $m$, $0 \le m \le n$, nonnegative real numbers $t_1 , \ldots , t_n$, and pairwise orthogonal rank one projections $P_1 , \ldots , P_n$ such that 
$$
A = \sum_{j=1}^m t_j P_j - \sum_{j=m+1}^n t_j P_j.
$$
Then $| A| = \sum_{j=1}^m t_j P_j + \sum_{j=m+1}^n t_j P_j$ and if we denote $S = \sum_{j=1}^m  P_j - \sum_{j=m+1}^n  P_j$, then $S$ is invertible, $S^2 = I$, and $A = S |A|$. Note that $|A|$ and $S$ commute. For an arbitrary $Z \in S_n$ the matrix $ZA + I = ZS |A| + I $ is invertible if and only if $|A|^{1/2} ZS |A|^{1/2} + I$ is invertible which is equivalent to the invertibility of $|A|^{1/2} Z |A|^{1/2} + S$.

Obviously, the matrix interval $[X, Y]$ is path-connected and therefore, all we need to verify is that for every $Z \in [X, Y]$ the matrix $|A|^{1/2} Z |A|^{1/2} + S$ is invertible. We have
\begin{equation}\label{puna}
|A|^{1/2} X |A|^{1/2} + S \le |A|^{1/2} Z |A|^{1/2} + S \le |A|^{1/2} Y |A|^{1/2} + S .
\end{equation}
It follows from $X,Y \in \U_A$ that there is a path connecting $|A|^{1/2} X |A|^{1/2} + S$ and $|A|^{1/2} 0 |A|^{1/2} + S = S$, as well as a path connecting $|A|^{1/2} Y |A|^{1/2} + S$ and $S$, both paths being contained in the set of invertible matrices in $S_n$. Since $S$ has $m$ positive eigenvalues and $n-m$ negative eigenvalues, the same is true for both 
$|A|^{1/2} X |A|^{1/2} + S$ and $|A|^{1/2} Y |A|^{1/2} + S$. Then the above remark and the left-hand side of (\ref{puna}) yield that $|A|^{1/2} Z |A|^{1/2} + S$ has at least $m$ positive eigenvalues, while 
the right-hand side of (\ref{puna}) implies that $|A|^{1/2} Z |A|^{1/2} + S$ has at least $n-m$ negative eigenvalues. It follows that $|A|^{1/2} Z |A|^{1/2} + S$ has $n$ nonzero eigenvalues and is therefore invertible, as desired.
\end{proof}

\begin{lemma}\label{frusti}
Let $n \ge 2$ be an integer and $A,B \in S_n$. Assume that for  every $P \in P_{n}^1$ there exist positive real numbers $s,t$ such that $B - (A+sP)$ is invertible and   $B - (A+tP)$ is of rank $n-1$. Then $A < B$. 
\end{lemma}

\begin{proof}
Lep $P$ be any projection of rank one. It is orthogonaly similar to $E_{11}$, the matrix whose all entries are zero but the $(1,1)$-entry that is equal to 1. It follows that there exist real numbers $c,d$ such that
$$
p(r) = \det (B - (A + rP)) = cr + d
$$
for every real $r$. We know that the linear function $p$ is not a constant function and therefore $c \not=0$. Moreover, it attains the zero value at some positive real number yielding that $d$ is nonzero, too. This tells that $p(0) \not=0$, or equivalently, $B-A$ is invertible.

Let $x \in \mathbb{R}^n$ be any unit vector. We then know that there exists a positive real number $t$ such that $B-A - t xx^t$ is singular, or equivalently, $I - t (B- A)^{-1}xx^t$ is not invertible. Then the trace of the rank one matrix $
 t (B- A)^{-1}xx^t$ must be equal to 1, which can be rewritten as
$$
\langle (B-A)^{-1} x , x \rangle = {1 \over t} > 0.
$$
Thus, $(B-A)^{-1} > 0$, and consequently, $B-A >0$.
\end{proof}

Let $n \ge 2$ be an integer and $G(n,2n)$ the Grassmann space consisting of all $n$-dimensional
subspaces of $\mathbb{R}^{2n}$. For two subspaces $\mathcal{M} , \mathcal{N} \in G(n,2n)$ and $k \in \{ 0, 1, \ldots , n \}$ we write
$$
d( \mathcal{M}, \mathcal{N}) = k
$$
if and only if $\dim \mathcal{M} \cap \mathcal{N} = n-k$ which is equivalent to $\dim \mathcal{M} + \mathcal{N} = n+k$.
Two elements $\mathcal{M} , \mathcal{N} \in G(n,2n)$
are adjacent if $d( \mathcal{M}, \mathcal{N}) = 1$.

To each point in $G(n, 2n)$, that is, to each $n$-dimensional
subspace $\mathcal{M}$ of $\mathbb{R}^{2n}$,
we associate an $n\times 2n$ matrix whose rows are coordinates 
of the vectors that
form a basis of $\mathcal{M}$. Each $n \times 2n$ matrix will be written in the block form $[ X \ Y]$, where both $X$ and $Y$ are square matrices.
It is trivial to see that
two such matrices $[ X \ Y]$
and $[ X' \ Y']$ correspond to the same subspace $\mathcal{M}$, that is, their rows represent two
bases of $\mathcal{M}$, if and only if $[ X \ Y] = Q [ X' \ Y']$ for some invertible 
$n\times n$ matrix $Q$. If this is the case, then either both $X$ and $X'$ are
invertible, or both $X$ and $X'$ are singular.

Let $[ X \ Y]$ be a matrix representig a point in $G(n,2n)$.
When $X$ is singular the corresponding point in the Grassmann space 
is called a point at infinity.
Otherwise, it is called a finite point.
Observe that a finite point $[ X \ Y]$
can be represented also with the matrix
$[ I \ X^{-1}Y ]$. 
The matrix $X^{-1}Y$ in such a representation
is uniquely determined by the finite point in the Grassmann space.

The proof of the following easy linear algebra result can be found in \cite{Wan}.

\begin{lemma}\label{kinez}
Let $n \ge 2$ be an integer and $\mathcal{M} , \mathcal{N} \in G(n,2n)$ two finite points with the corresponding canonical matrix representations $[ I \ T]$ and $[ I \ S]$, respectively. Then
$$
d( \mathcal{M}, \mathcal{N}) = {\rm rank}\, (T-S).
$$
\end{lemma}

In particular, if $\mathcal{M}$ and $\mathcal{N}$ are two finite points
in the Grassmann space, represented with
the uniquely determined square matrices $T$ and $S$, respectively, then 
$\mathcal{M}$ and $\mathcal{N}$ are adjacent if and only if 
${\rm rank}\, (T-S) = 1$. Two such matrices $T$ and $S$ are called adjacent.

A subspace $\mathcal{M}$ corresponding to a matrix $W = [ X \ Y]$ is self-dual if $WKW^t = 0$, where $K$ is the $2n \times 2n$ matrix
$$
K= \left[ \begin{matrix} 0 & I \cr -I & 0 \cr \end{matrix}\right]
$$
(note that this definition is not sensitive to the choice of a matrix $W$ representing the subspace $\mathcal{M}$). In other words, $W = [ X \ Y]$ represents a self-dual space if and only if $XY^t = YX^t$. In particular, if $\mathcal{M}$ is a self-dual finite point than it has a unique respresentation $[ I \, X]$ with $X$ being a symmetric matrix. 

\begin{proof}[Proof of Proposition \ref{predmain}]

We first need to see that $\psi_A$ is well-defined, that is, we have to verify that $\psi_A (X)$ is a symmetric matrix for every $X \in \mathcal{W}_A$.  Indeed, we have
$$
\psi_A (X) = (XA + I)^{-1} X = (XA + I)^{-1} X (AX+ I) (AX + I)^{-1} 
$$
$$
= (XA + I)^{-1}  (X A+ I) X (AX + I)^{-1} = X (AX + I)^{-1} .
$$
Our next goal is to verify that $\psi_A (X) \in \mathcal{W}_{-A}$ for every $X \in \mathcal{W}_A$. This follows from
$$
-A \psi_A (X) + I = - AX (AX+I)^{-1} + I $$  $$= - AX (AX+I)^{-1} + (AX + I) (AX+I)^{-1} = (AX+I)^{-1}.
$$
Finally, for every $X \in \mathcal{W}_A$ we have
$$
\psi_{-A} ( \psi_A (X)) = \psi_A (X) (-A \psi_A (X) + I)^{-1} = 
 X (AX + I)^{-1} (AX + I) = X.
$$
Similarly, $\psi_A ( \psi_{-A} (Y)) = Y$ for every $Y \in \mathcal{W}_{-A}$. This completes the proof of the first item. The second one is trivial. It is then clear that $\psi_A$ maps  $\U_A$ bijectively onto $\U_{-A}$. Thus, it reamins to show that for any pair $X,Y \in \U_A$ we have
$$
X \le Y \Rightarrow \psi_A(X) \le \psi_A (Y).
$$
By continuity, it is enough to check that 
for any pair $X,Y \in \U_A$ we have
\begin{equation}\label{ovink}
X < Y \Rightarrow \psi_A(X) <  \psi_A (Y).
\end{equation}

We will first prove that for every pair $X,Y \in \U_A$ we have 
\begin{equation}\label{avavav}
{\rm rank}\, ( \psi_A (Y) - \psi_A (X)) = {\rm rank}\, (Y-X).
\end{equation}
Clearly, the $2n \times 2n$ matrix 
$$
M = \left[ \begin{matrix} I & 0 \cr A & I \cr \end{matrix}\right]
$$
is invertible. Let $X,Y \in \U_A$ and denote $r = {\rm rank}\, (Y-X)$.
Denote by $\mathcal{M} , \mathcal{N} \in G(n,2n)$ the finite points with the matrix representations $[ I \ X]$ and $[ I \ Y]$, respectively. Then
$$
d( \mathcal{M}, \mathcal{N}) = r.
$$
We further denote by
$\mathcal{M}' , \mathcal{N}' \in G(n,2n)$ the subspaces with the matrix representations $[ I \ X]M$ and $[ I \ Y]M$, respectively.
Because $M$ is invertible we have
$$
d( \mathcal{M}', \mathcal{N}') = r.
$$
Obviously, the matrix representations of $\mathcal{M}' , \mathcal{N}'$ are $[ XA + I \ \ X]$ and $[ YA + I \ \ Y]$, respectively.
As we know that matrices $XA + I$ and $YA + I$ are both invertible the subspaces   $\mathcal{M}' , \mathcal{N}'$ are represented also by $[ I \ \  (XA + I)^{-1}  X]$ and $[ I \ \  (YA + I)^{-1} Y]$, respectively.  
This finally implies that ${\rm rank}\, ( \psi_A (Y) - \psi_A (X))  = r$ which completes the verification of (\ref{avavav}).

Let now $X$ be any element of $\U_A$ and $P$ any projection of rank one. Then we know that $X + tP \in \U_A$ for all real numbers $t$ whose absolute value is small enough.
Let $\mathcal{I} \subset \mathbb{R}$ be any interval such that $0 \in \mathcal{I}$ and  $X + tP \in \U_A$ for all $t \in \mathcal{I}$. Using (\ref{avavav}) we see that for every nonzero 
$t \in \mathcal{I}$ the difference $\psi_A (X +tP) - \psi_A (X)$ is of rank one, that is, for every nonzero
$t \in \mathcal{I}$ there exist a uniquely determined nonzero real number $f(t)$ and a uniqeuely determined rank one projection $Q_t$ such that $\psi_A (X + tP) = \psi_A (X) + f(t) Q_t$. We set $f(0) =0$. Now, if $t,s \in \mathcal{I}$ are two distinct nonzero real numbers then
applying (\ref{avavav}) once more we see that $ f(t) Q_t - f(s) Q_s$ is of rank one. Because both $f(t)$ and $f(s)$ are nonzero this is possible only if $Q_s = Q_t$. Thus, $Q = Q_t$ is independent of $t \in \mathcal{I}$ and we have 
$\psi_A (X + tP) = \psi_A (X) + f(t) Q$, $t \in \mathcal{I}$. Obviously, $f : \mathcal{I} \to \mathbb{R}$ is injective and continuous and therefore $f$ is either strictly increasing function, or strictly decreasing function. 
Because $\psi_{-A}$ has the same properties as $\psi_A$ the map $\tau$ which associates to each $P \in P_{n}^1$ the projection $Q$ as described above is bijective. The set of rank one projections $P$ for which we have $\psi_A (X +tP) = \psi_A(X) +  f(t)\tau (P)$ for all real $t$ with absolute value small enough with $f$ being strictly increasing is obviously closed in $P_{n}^1$. The same is true for the
set of rank one projections $P$ for which we have $\psi_A (X +tP) = \psi_A (X) + f(t)\tau (P)$ for all real $t$ with absolute value small enough with $f$ being strictly decreasing. It is well-known that $P_{n}^1$ is a connected manifold. Thus 
$P_{n}^1$ is equal to one of these two sets.

Assume next that $X,Y \in \U_A$ and $X < Y$. By Lemma \ref{uzivee} we know that for  all $P \in P_{n}^1$ and $t \in  [0, \alpha(X,P; Y) ]$   we have $X + tP \in \U_A$.

We have two possibilities. The first one is that there exists a bijective map $\tau : P_{n}^1  \to P_{n}^1$ such that for every $P \in P_{n}^1$ there exists a continuos strictly increasing function $f_P : [0, \alpha(X,P; Y) ] \to [0, \infty)$ such that we have
$$
\psi_A (X + tP) = \psi_A (X) + f_P (t) \tau (P), \ \ \ t \in  [0, \alpha(X,P; Y) ], \ \, P \in P_{n}^1.
$$
The second one is that there exists a bijective map $\tau : P_{n}^1  \to P_{n}^1$ such that for every $P \in P_{n}^1$ there exists a continuos strictly decreasing function $f_P : [0, \alpha(X,P; Y) ] \to (- \infty , 0 ]$ such that we have
$\psi_A (X + tP) = \psi_A (X) + f_P (t) \tau (P)$ for all $t \in  [0, \alpha(X,P; Y) ]$ and $P \in P_{n}^1$.

We start with the first possibility. For every $Z \in S_n$, $X < Z \le Y$, and every $P \in P_{n}^1$ we have
$$
{\rm rank}\, (\psi_A (Z)-  ( \psi_A (X) + f_P (\alpha(X,P; Z)) \tau (P))) = n-1.
$$
It follows from Lemma \ref{frusti} that $\psi_A (Z) > \psi_A (X)$. Moreover, we have 
$$
\alpha (\psi_A (X), \tau (P) ; \psi_A (Z)) = f_P (\alpha(X,P; Z)).
$$
This together with Lemma 
\ref{less}
implies that $\psi_A$ maps $( X, Y] = \{ Z \in S_n \, : \, X < Z \le Y \}$ to 
$( \psi_A (X), \psi_A (Y)]$ and preserves order. By the continuity,
$$
\psi_A ( [ X, Y] ) \subset
[ \psi_A (X), \psi_A (Y)],
$$
and for every pair $Z,W \in [ X, Y]$ we have
$$
Z < W \Rightarrow \psi_A (Z) < \psi_A (W).
$$

In the same way we see that in the second case 
$$\psi_A ( [ X, Y] ) \subset
[ \psi_A (Y), \psi_A (X)],
$$
and for every pair $Z,W \in [ X, Y]$ we have
$$
Z < W \Rightarrow \psi_A (Z) > \psi_A (W).
$$

It follows from Lemma \ref{trivialn} that either we have the first possibility for every pair $X,Y \in \U_A$ satisfying $X < Y$, or we have the second possibility for every pair $X,Y \in \U_A$ satisfying $X < Y$. We will complete the proof if we find a pair $X,Y \in \U_A$ satisfying $X < Y$ such that $\psi_A(X) < \psi_A (Y)$.

Because $0 \in \U_A$ and $\U_A$ is open we can find a positive real number $a$ such that $aI, (2a)I \in \U_A$ and ${ 1 \over 2a} I + A >0$. Then
$$
0 < {1 \over 2a} I + A <{1 \over a} I + A,
$$
and therefore
$$
\psi_A (aI) = ({1 \over a} I + A)^{-1} < ({1 \over 2a} I + A)^{-1} = \psi_A ((2a)I).
$$
This completes the proof.
\end{proof}

In the rest of the section we will assume that $n \ge 2$. 

\begin{lemma}\label{MZ}
Let $T$ be an invertible real $n \times n$ matrix. Then for every $X  \in [0,I]$ the matrix $X (T^t T -I) +I$ is invertible and the map $\phi_T :  [0,I] \to S_n$ given by 
\begin{equation}\label{lpy}
\phi_T (X) =  T \left( X (T^t T -I) +I \right)^{-1} X T^t, \ \ \
X \in [0,I],
\end{equation}
is an order embedding which maps $[0,I]$ bijectively onto itself. If $S$ is another invertible real $n \times n$ matrix then 
\begin{equation}\label{mmaay}
\phi_{ST} (X) = \phi_S (\phi_T (X))
\end{equation}
for every $X \in [0,I]$.
\end{lemma}

This lemma has been known, see the first part of the proof of Theorem 3.1 in \cite{Sem}. But we can also deduce it from our more general results. Indeed, it is trivial to verify that $\phi_T(0) = 0$ and  $\phi_T(I) = I$. By Lemma \ref{uzivee}, 
for every $X  \in [0,I]$ the matrix $X (T^t T -I) +I$ is invertible. The equality (\ref{mmaay}) can be verified by a straightforward calculation. The verification of the lemma can now be easily completed using Proposition \ref{predmain}.

We can summarize the above statement by saying that the map $\phi_T$ given by (\ref{lpy}) is an order automorphism of 
$[0,I]$ and the map $T \mapsto \phi_T$ is a homomorphism of the general linear group $GL(n, \mathbb{R})$ to a group of order automorphisms of the matrix interval $[0,I]$. 

If $T$ is an invertible real $n \times n$ matrix then clearly 
$$
\phi_T ((1/2)I) = (I+S)^{-1}
$$
where $S = (TT^t)^{-1}$.  When $T$ runs over the set of all invertible real $n \times n$ matrices, $S$ runs over all positive definite matrices. It follows that for every $C \in (0,I)$ there exists an invertible matrix $T$ such that $\phi_T ((1/2)I) = C$. This further implies that for every pair $C,D \in (0,I)$ there exists an invertible $n \times n$ matrix $R$ such that $\phi_R (C) = D$. 

If $O$ is an orthogonal matrix then $\phi_O (X) = OXO^t$ for every $X \in [0,I]$.

We are now ready to prove the following.

\begin{corollary}\label{crev}
Let $n \ge 2$ and $\phi : [0,I] \to S_n$ be an order embedding satisfying $\phi (0) = 0$ and $\phi (I) = I$. Assume that $\phi$ is continuous at $0$ and $I$. Then there exists an invertible real $n \times n$ matrix $T$ such that $\phi = \phi_T$.
\end{corollary}

\begin{proof}
Using Lemma \ref{automcont} and the fact that for every pair $C,D \in (0,I)$ there exists an invertible $n \times n$ matrix $R$ such that $\phi_R (C) = D$
we can find invertible matrices $S_1$ and $S_2$ such that the map $X \mapsto \phi_{S_1} (\phi (\phi_{S_2} (X)))$ sends matrices $0, (1/2)I,I$ to themselves and is continuous at $0$, $(1/2)I$, and $I$. Applying Propositions \ref{nice} and \ref{nice2} we see that there exists an orthogonal matrix $O$ such that $\phi_{S_1} \circ \phi \circ \phi_{S_2} = \phi_O$, or equivalently, $\phi = \phi_T$, where $T = {S_{1}^{-1}OS_{2}^{-1}}$.
\end{proof}

\begin{corollary}\label{zrev}
Let $n \ge 2$ and $\phi : [0,I] \to S_n$ be an order embedding satisfying $\phi (0) = 0$. Assume that $\phi$ is continuous at $0$ and $I$. Then for every unit vector $x \in \mathbb{R}^n$ there exist a unit vector $u$ and real numbers $a,b$ satisfying $a > 0$ and $b > -1$ such that 
$$
\phi (txx^t) = { at \over bt+1} uu^t
$$
for every $t \in [ 0, 1] $.
\end{corollary}

\begin{remark}
A careful reader has observed that already the proof of Proposition \ref{predmain} yields that  for every unit vector $x \in \mathbb{R}^n$ there exist a unit vector $u$ and a strictly increasing function $f :  [ 0, 1] \to \mathbb{R}$ such that
$\phi (txx^t) = f(t) uu^t$
for every $t \in [ 0, 1] $. In the special case when the order embedding $\phi$ is defined on the matrix interval  $[0,I]$ and satisfies $\phi (0) = 0$ the above statement gives a more concrete description of such a function $f$.
\end{remark}

\begin{proof}
The proof is by a straightforward calculation. We start with the special case that $\phi (I) = I$. 
We know that there is an invertible real matrix $T$ such that $\phi (X) = T ( X (T^t T -I) + I)^{-1}XT^t$, $X \in [0,I]$. Denoting $A = A^t = T^t T -I  > -I$ we see that for every real $t$, $0 < t \le 1$, we have $\langle Ax, tx \rangle = t \langle Ax, x \rangle > t  \langle (-I)x, x \rangle = -t \ge -1$. It follows that
$$
\phi (t xx^t) = tT( txx^t A + I)^{-1} xx^t T^t = tT( (tx)(Ax)^t  + I)^{-1} x(Tx)^t   
$$
$$
= tT \left( I - { 1 \over 1 + t\langle Ax,x \rangle} (tx)(Ax)^t \right)x (Tx)^t = t (Tx)(Tx)^t - { t^2 \langle Ax, x \rangle \over 1 + t \langle Ax, x \rangle } (Tx)(Tx)^t
$$
$$
= { t \over 1 +  t\langle Ax,x \rangle} (Tx)(Tx)^t.
$$
Obviously, $\langle Ax,x \rangle = \| Tx \|^2 - \| x \|^2 = \| Tx \|^2 - 1$. If we denote $u = {1 \over \| Tx \|} Tx$, then 
$$
\phi (t xx^t) = { t \| Tx \|^2 \over t(\|Tx \|^2 - 1 ) + 1 } uu^t =  { at \over bt+1} uu^t,
$$
where $a > 0$, and consequently $a-1 > -1$.

The general case can be reduced to the above special case by considering the map $X \mapsto \phi (I)^{-1/2} \phi (X)  \phi (I)^{-1/2}$, $X \in  [0,I]$.
\end{proof}

\begin{remark}\label{ysr} 
Let $a,b$ be positive real numbers and $c,d > -1$. We define the function $f_{a,c} :  [ 0, 1] \to  [ 0, 1]$ by $f_{a,c}(t) =  { at \over ct+1}$, $t \in [ 0, 1]$. If  $\delta$ is a positive real number no larger than $1$, and $f_{a,c}$ coincides with $f_{b,d}$ on the interval $[0, \delta]$, then $a=b$ and $c=d$. The verification of this fact is trivial.
\end{remark}
\begin{corollary}\label{verc}
Let $n \ge 2$ and $\phi : [0,I] \to S_n$ be an order embedding. Assume that $\phi$ is continuous at $0$ and $I$. Then $\phi$ is continuous.
\end{corollary}

\begin{proof}
Replacing $\phi$ by the map $X \mapsto \phi (X) - \phi (0)$, $X \in  [0,I]$, we may assume with no loss of generality that $\phi (0) =0$. We have $\phi (I) \ge 0 = \phi (0)$ and by Lemma \ref{nicce} we know that actually $\phi (I) >0$. Replacing $\phi$ once more, this time with the map $X \mapsto \phi(I)^{-1/2} \phi (X) \phi (I)^{-1/2}$, $X \in  [0,I]$, we arrive at the situation when $\phi(I) =I$. By Corollary \ref{crev}, $\phi$ is continuous.
\end{proof}

Let $A, B \in S_n$ with $A < B$. Then the map $X \mapsto (B-A)^{1/2} X (B-A)^{1/2} +A$, $X \in [0,I]$, is an order embedding of $[0,I]$ onto $[A,B]$. Clearly, it is a homeomorphism. Thus, we have the following straightforward consequence of Corollary \ref{verc}.

\begin{corollary}\label{verc2}
Let $n \ge 2$ and $A, B \in S_n$ with $A < B$. Assume that
 $\phi : [A,B] \to S_n$ is an order embedding that is continuous at $A$ and $B$. Then $\phi$ is continuous.
\end{corollary}

\begin{proposition}\label{aqaq}
Let $\U \in S_n$ be a domain and $\phi : \U \to S_n$ an order embedding. Then $\phi$ is continuous.
\end{proposition}

\begin{proof}
Take any $X_0 \in \U$. As $\U$ is open there exists $\varepsilon >0$ such that $(X_0 - \varepsilon I, X_0 + \varepsilon I) \subset \U$. By an easy and straightforward modification of Lemma \ref{automcont} the map $\phi$ is continuous 
at every but countably many points of the form $X_0 + tI$, $-\varepsilon < t < \varepsilon$. In particular, there exist positive real numbers $a,b < \varepsilon$ such that $\phi$ is continuous at $X_0 - aI$ and $X_0 +bI$ and by Corollary \ref{verc2},
$\phi$ is continuous at $X_0$. 
\end{proof}

If $\U_1 , \U_2 \subset S_n$ are domains and $\phi_j : \U_j \to S_n$, $j=1,2$, order embeddings then we say that $\phi_2$ is an extension of $\phi_1$ if $\U_1 \subset \U_2$ and the restriction of $\phi_2$ to $\U_1$ coincides with $\phi_1$. An order embedding $\phi : \U \to S_n$ is maximal if there is no proper extension of $\phi$. By Zorn's Lemma, each order emebedding has a maximal extension. 

\begin{lemma}\label{kakjihje}
Let $n \ge 2$ and $\phi_1 , \phi_2 : [0,I] \to S_n$ be order embeddings satisfying $\phi_j (0) = 0$, $j=1,2$. Assume that $\phi_1$ and $\phi_2$ are continuous at $0$ and $I$. 
Let $C \in (0,I)$ and assume that $\phi_1 (X) = \phi_2 (X)$ for every $X \in  [0,C]$. Then $\phi_1 = \phi_2$.
\end{lemma}

\begin{proof}
It follows from Corollary \ref{zrev} and Remark \ref{ysr} that
$$
\phi_1 (tP) = \phi_2 (tP)
$$
for every $P\in P_{n}^1$ and every real $t$, $0 \le t \le 1$. 

Let $X \in [0,I]$ and $P \in P_{n}^1$.
We know that for every nonnegative real $t$ such that $tP \le \phi_1 (X)$ there exist  $Q\in P_{n}^1$ and real $s$, $0 \le s \le 1$, such that $\phi_1 (sQ) = tP$. 
It follows that $sQ \le X$, and consequently, $tP \le \phi_2 (X)$. The same conclusion is reached by swaping the roles of $\phi_1$ and $\phi_2$. Thus, for every nonnegative real number $t$ we have
$$
tP \le \phi_1 (X) \iff  tP \le \phi_2 (X).
$$
It follows from Corollary \ref{equal1} that $\phi_1 (X) = \phi_2 (X)$.
\end{proof}

We have a straightforward consequence.

\begin{corollary}\label{kakjihje2}
Let $n \ge 2$, $A,B \in S_n$ with $A<B$, and $\phi_1 , \phi_2 : [A,B] \to S_n$ be order embeddings that are continuous at $A$ and $B$. 
Let $C \in (A,B)$ and assume that $\phi_1 (X) = \phi_2 (X)$ for every $X \in  [A,C]$. Then $\phi_1 (X) = \phi_2 (X)$ for every $X \in  [A,B]$.
\end{corollary}

In the same way we prove the following.

\begin{lemma}\label{kakjihje3}
Let $n \ge 2$, $A,B \in S_n$ with $A<B$, and $\phi_1 , \phi_2 : [A,B] \to S_n$ be order embeddings that are continuous at $A$ and $B$. 
Let $C \in (A,B)$ and assume that $\phi_1 (X) = \phi_2 (X)$ for every $X \in  [C,B]$. Then $\phi_1 (X) = \phi_2 (X)$ for every $X \in  [A,B]$.
\end{lemma}

\begin{corollary}\label{kakjihje4}
Let $n \ge 2$, $A,B \in S_n$ with $A<B$, and $\phi_1 , \phi_2 : [A,B] \to S_n$ be order embeddings that are continuous at $A$ and $B$. 
Let $C,D \in S_n$ with $A < C < D < B$ and assume that $\phi_1 (X) = \phi_2 (X)$ for every $X \in  [C,D]$. Then $\phi_1 (X) = \phi_2 (X)$ for every $X \in  [A,B]$.
\end{corollary}

\begin{proof}
By Corollary \ref{verc2}, both $\phi_1$ and $\phi_2$ are continuous on $ [A,B]$. It follows that we can apply Lemma \ref{kakjihje3} for the interval $ [A,D]$ with $C \in (A,D)$ to conclude that $\phi_1$ and $\phi_2$ coincide on $ [A,D]$. Using Corollary \ref{kakjihje2} we get $\phi_1 = \phi_2$.
\end{proof}

\begin{corollary}\label{jjmj}
Let $n \ge 2$, $\U \subset S_n$ be a matrix domain, $A,B,C,D \in \U$ with $A< B$, $C<D$, and $ [A,B],  [C,D] \subset \U$, and $\phi_1 , \phi_2 : \U \to S_n$ be order embeddings. Assume that the intersection $ [A,B] \cap  [C,D]$ has a nonempty interior and that $\phi_1 (X) = \phi_2(X)$ for every $X \in  [A,B]$. Then $\phi_1 (X) = \phi_2(X)$ for every $X \in  [C,D]$.
\end{corollary}

\begin{proof}
By the assumption on a nonempty interior of the intersection of matrix intervals we know that there exist $E,F \in \U$ such that $(E,F) \subset   [A,B] \cap  [C,D]$. We can futher find $M,N \in \U$ such that $E < M < N < F$. Then 
$C \le E < M < N < F \le D$ and because $\phi_1 (X) = \phi_2(X)$ for every $X \in  [M,N]$ we deduce from Corollary \ref{kakjihje4} that $\phi_1$ and $\phi_2$ coincide on $ [C,D]$.
\end{proof}

A straightforward consequence of Corollary \ref{jjmj} and Lemma \ref{trivialn} is the following identity theorem.

\begin{theorem}\label{mmjm}
Let $n \ge 2$, $\U \subset S_n$ be a matrix domain, and $\mathcal{W} \subset \U$ a nonempty open subset. Assume that
$\phi_1 , \phi_2 : \U \to S_n$ are order embeddings such that $\phi_1 (X) = \phi_2(X)$ for every $X \in  \mathcal{W}$. Then $\phi_1 = \phi_2$.
\end{theorem}

The following technical lemma is easy and has been proved in the complex case in \cite[Lemma 4.2]{MoS}. Exactly the same proof works also in the real case.

\begin{lemma}\label{diverge}
Let $A \in S_n$ and let $\mathcal{W}_A$ be defined as in Introduction. Define $\theta_A : \mathcal{W}_A \to S_n$ by $\theta_A (X) = (XA + I)^{-1}X$, $X \in \mathcal{W}_A$.
Suppose that a sequence $(X_n)\subset \mathcal{W}_A$ converges to a matrix $X\in S_n \setminus \mathcal{W}_A$. 
Then we have $\|\theta_A(X_n)\|\to \infty$.
\end{lemma}

We will actually need the following straightforward consequence.

\begin{corollary}\label{kdjeno}
Let $A \in S_n$ and let $\mathcal{U}_A$ be defined as in Introduction. Define $\theta_A : \mathcal{U}_A \to S_n$ by $\theta_A (X) = (XA + I)^{-1}X$, $X \in \mathcal{U}_A$.
Suppose that a sequence $(X_n)\subset \mathcal{U}_A$ converges to a matrix $X\in S_n \setminus \mathcal{U}_A$. 
Then we have $\|\theta_A(X_n)\|\to \infty$.
\end{corollary}

\section{Proof of the main result}

\begin{proof}[Proof of Theorem \ref{main}]
By Proposition \ref{aqaq}, the map $\phi$ is continuous. Since $\U$ is open, there exists a positive integer $m$ such that $[0, (1/m)I ] \subset \U$. Set $\U' = \{ mX \, : \, X \in \U \}$ and define $\phi' : \U' \to S_n$ by $\phi' (X) = \phi ((1/m)X)$, $X \in \U'$. Clearly, $[0,I ] \subset \U'$.

For an invertible matrix $T$ and a symmetric matrix $A$ we define the map $\psi_{T,A}$ by
$$
\psi_{T,A} (X) = T(XA + I)^{-1} X T^t , \ \ \ X \in \U_A .
$$
Let $a$ be a positive real number. Then for every symmetric matrix $X$ such that $aX \in \U_A$ we have
$$
\psi_{T,A} (aX) = \psi_{\sqrt{a}T, aA} (X).
$$

It follows that we can assume with no loss of generality that $[0,I ] \in \U$.  Moreover, by Lemma \ref{nicce}, we have $\phi (I) = S >0$. Replacing $\phi$ by the map $X \mapsto S^{-1/2} \phi (X) S^{-1/2}$ we may 
further assume that $\phi (I) = I$. Using Corollary \ref{crev} we conclude that the restriction of $\phi$ to $[0,I ]$ coincides with $\psi_{T,A}$, where $A = T^t T -I$.

We consider the restriction of $\phi$ to $(0,I)$. It follows from Corollary \ref{kdjeno} that  $\psi_{T,A} : \U_A \to S_n$ is a maximal extension of this restriction. We will show that it is unique. Once we will verify this claim the proof will be completed.

Assume on the contrary that there exist a matrix domain $\mathcal{V} \subset S_n$ and an order emedding $\xi : \mathcal{V} \to S_n$ such that $(0,I) \subset \mathcal{V}$, 
 the restriction of $\xi$ to $(0,I)$ coincides with $\psi_{T,A}$, and $\mathcal{V} \not\subset \U_A$. We denote by $\U_0$ the component of $\mathcal{V} \cap \U_A$ that contains $(0,I)$. By Theorem \ref{mmjm} we have
$\psi_{T,A} (X) = \xi (X)$ for every $X \in \U_0$. Take any $E \in \mathcal{V} \setminus \U_A$. Then there exsists a path $\gamma : [0,1 ] \to \mathcal{V}$ such that $\gamma (0) = (1/2)I$ and $\gamma (1) = E$. Set $t_0 = \inf \{ t \in   [0,1 ] \, : \, \gamma (t) \not\in \U_0 \} = \min \{ t \in   [0,1 ] \, : \, \gamma (t) \not\in \U_0 \} > 0$. Set $E_0 = \gamma (t_0) \not\in \U_A$. Take any sequence $(t_k) \subset [0, t_0)$ with $\lim t_k = t_0$. It follows from Corollary \ref{kdjeno} that $\| \xi ( \gamma (t_k)) \| = \| \psi_{T,A} (\gamma (t_k )) \| \to \infty$, contradicting Proposition \ref{aqaq}.
\end{proof}

A careful reader has noticed that as a direct byproduct of the above proof we have the following consequence.

\begin{corollary}\label{glavposl}
Let $n \ge 2$ be an integer, $\U \subset S_n$ a matrix domain and $\phi: \U \to S_n$ an order embedding. Then there exists a unique maximal extension $\psi : \mathcal{W} \to S_n$ of order embedding $\phi$. Moreover, if $\mathcal{W} \not= S_n$, then for every 
sequence $(X_n)\subset \mathcal{W}$ converging to a matrix $X\in S_n \setminus \mathcal{W}$ 
we have $\|\psi (X_n)\|\to \infty$.
\end{corollary}

\section{Final remarks}

\begin{enumerate}
\item Let us start by explaining what the main theorem tells for order emebeddings of open matrix intervals. We have four types of open matrix intervals. These are $(A,B)$, $(A, \infty) = \{ X \in S_n \, : \, X > A \}$, 
$(-\infty , A) = \{ X \in S_n \, : \, X < A \}$, and $(-\infty , \infty ) = S_n$. Here, $A,B \in S_n$ is any pair of symmetric matrices satisfying $A < B$. If two open matrix intervals $\mathcal{I}$ and $\mathcal{J}$ are order isomorphic and $\varphi : \mathcal{I} \to \mathcal{J}$ is an order isomorphism, then every order embedding $\phi : \mathcal{I} \to S_n$ is of the form $\phi = \phi' \circ \varphi$, where $\phi'$ is an order embedding of $\mathcal{J}$. Consequently, if we know the general form of order embeedings of $\mathcal{J}$, then we have a complete descritpion of order embeddings of $\mathcal{I}$. Let $A,B \in S_n$ with $A < B$. Then the map $X \mapsto 
{1 \over 2} (B-A)^{1/2} X (B-A)^{1/2} + {1 \over 2}(A+B)$ is an order isomorphism of $(-I,I)$ onto $(A,B)$. Clearly, $X \mapsto X-A$ is an order automorphism of $(A, \infty)$ onto $(0, \infty)$. Similarly, each interval $(-\infty, A)$ is via a translation order isomorphic to $(-\infty , 0)$. The map $X \mapsto - X^{-1}$ is an order isomorphism of $(0, \infty)$ onto $(-\infty , 0)$. And finally, the map $X \mapsto (I-X)^{-1} - {1 \over 2}I$ is an order isomorphism of $(-I,I)$ onto $(0, \infty)$.

Thus, every open matrix interval $\mathcal{I} \not= (-\infty , \infty)$ is order isomorphic to $(-I,I)$ and we we know how to construct an order isomorphism of $\mathcal{I}$ onto $(-I,I)$. Therefore we know exactly what are all order embeddings of any open matrix interval once we have a complete description of all order embeddings of matrix intervals $(-I, I)$ and $S_n$.

\begin{theorem}\label{openefekt}
Let $n$ be an integer, $n \ge 2$, and $\phi : (-I,I) \to S_n$ a map. Then $\phi$ is an order embedding if and only if there exist a real invertible $n \times n$ matrix $T$ and $A,B \in S_n$ satisfying $\| A \| \le 1$ such that
$$
\phi (X) = T(XA + I)^{-1} X T^t + B, \ \ \ X \in ( -I , I).
$$
\end{theorem}

\begin{proof}
If $T,A,B$ are as above then for every $X \in ( -I,I)$ we have $\| XA \| \le \| X \| \, \| A \| < 1$, and therefore, $XA + I$ is invertible. It follows that the map $\phi$ is an order embedding of $(-I,I)$.

To prove the other direction we assume that $\phi$ is an order embedding. Replacing the map $\phi$ by the map $X \mapsto \phi (X) - B$, $X \in (-I,I)$, where $B = \phi (0)$, we can assume with no loss of generality that $\phi (0) = 0$. Using our main theorem we see that the map $\phi$ is of the desired form for some invertible $T$ and some $A$ with the property that $XA + I$ is invertible for every $X \in (-I,I)$. All we need to do is to verify that $\| A \| \le 1$. Assume that this is not true. Then, after identifying matrices with operators and a suitable change of orthonormal bases we may assume that $A$ is diagonal, $A = {\rm diag}\, (a_1 , \ldots , a_n)$ with $a_1, \ldots , a_n \in \mathbb{R}$ and $| a_1 | > 1$. 
Then $X = {\rm diag}\, \left (-{ 1 \over a_1} ,  0, \ldots , 0 \right) \in (-I,I)$ and $XA + I$ is not invertible, a contradiction.
\end{proof}

\begin{theorem}\label{vseskupaj}
Let $n$ be an integer, $n \ge 2$, and $\phi : S_n \to S_n$ a map. Then $\phi$ is an order embedding if and only if there exist a real invertible $n \times n$ matrix $T$ and $B \in S_n$ such that
$$
\phi (X) = T X T^t + B
$$
for every $X \in S_n$.
\end{theorem}

\begin{proof}
One direction is trivial and to prove the other direction we use the main theorem to see that there exist a real invertible $n \times n$ matrix $T$ and $A,B \in S_n$ such that for every $X \in S_n$ the matrix $XA + I$ is invertible and
$$
\phi (X) = T(XA + I)^{-1} X T^t + B.
$$
The condition that $XA + I$ is invertible for every $X \in S_n$ implies that $A=0$, as desired.
\end{proof}

\item   
The description of order automorphisms of the set of all self-adjoint operators on a Hilbert space was obtained in \cite{Mo1}.
A shorter proof given in \cite{Se0} was based on the following observation.
Two matrices $A,B \in S_n$, $A\not=B$, are adjacent if and only if the following condition is staisfied:
\begin{itemize}
\item $A$ and $B$ are comparable and any two matrices $C,D \in S_n$ that are in between $A$ and $B$ are comparable.
\end{itemize}
Actually, this equivalence was proved for hermitian matrices but the same arguments work also for real symmetric matrices. Let us outline the proof. Assume first that $A$ and $B$ are adjacent. Any symmetric matrix of rank one is of the form $tP$ for some nonzero real number $t$ and some rank one projection $P$. Hence, $B-A = t_0 P$ for some nonzero real number $t_0$  and some rank one projection $P$, and then $B\ge A$ if $t_0 > 0$, and $B\le A$ if $t_0 < 0$. That is, $A$ and $B$ are comparable. We have two possibilities and let us consider just one of them, say, $t_0 > 0$ and $B = A + t_0 P\ge A$. It is then easy to see that any matrix $X$ that is in between $A$ and $B$, that is, $A \le X \le A + t_0 P$, is of the form $X = A + tP$, $0 \le t \le t_0$. Hence, if $C$ and $D$ are in between $A$ and $B$, then $C = A +tP$ and $D = A + sP$ for some real numbers $s,t \in [0,t_0]$ and then either $C \le D$ or $D \le C$ depending on whether $t \le s$ or $s\le t$.

To prove the converse we assume that $A$ and $B$, $A \not= B$, are not adjacent. If they are not comparable, then we are done. If they are comparable, then we have two possibilities and again we will consider just one of them. So, assume that $A \le B$. Then it is easy to guess and not difficult to prove that the condition ${\rm rank}\, (B-A) \ge 2$ yields the existence of $C,D \in  S_n$, $A \le C,D \le B$, such that neither $C \le D$ nor $D \le C$. 

We have characterized adjacency with the use of Loewner's order and therefore if a certain bijective map preserves Loewner's order then it also preserves adjacency. Describing the general form of adjacency preserving maps on various matrix spaces has been a classical problem, see \cite{Hu1, Hu2, Hu3, Hu4, Hu5, Hu6, Hu7, Hu8, HuS, Se3, Wan} and the references therein. 

\item
In the previous remark we have explained the connection between order embeddings (order isomorphisms) and adjacency preserving maps. Now we will show that adjacency preserving maps are closely related to the fundamental theorem of chronogeometry. Let us first recall that two matrices $A$ and $B$ are said to be coherent if ${\rm rank}\, (A-B) \le 1$. Thus, $A$ and $B$ are coherent if and only if $A=B$ or $A$ and $B$ are adjacent.

Let $\mathcal{M} = \{ (x,y,z,t)\, : \, x,y,z,t \in \mathbb{R} \}$ be the classical Minkowski space of all spacetime events.
Two spacetime events $r_1 = (x_1 , y_1 , z_1 , t_1)$ and $r_2 = (x_2 , y_2 , z_2 , t_2)$
are lightlike if the light signal can pass between $r_1$ and $r_2$, that is
$$
(x_2 - x_1 )^2 + (y_2 - y_1 )^2 + (z_2 - z_1 )^2 = c^2 (t_2 - t_1 )^2 .
$$
In the mathematical foundations of special relativity we usually make a harmless normalization assuming that $c=1$.
In 1950 A.D. Aleksandrov \cite{Al1}, see also \cite{Al2, Al3}, proved the fundamental theorem of chronogeometry stating that bijective preservers of lightlikeness on $\mathcal{M}$
 are Lorenz transformations up to a scalar factor and a translation.
It has been known for a long time that the problem of describing the general form of
bijective  lightlikeness preserving maps on $\mathcal{M}$ and the
problem of characterizing bijective adjacency preserving
maps on the space of $2 \times 2$ complex hermitian matrices are equivalent. Indeed, let us
associate to each spacetime event $(x,y,z,t)$ the $2\times 2$ hermitian
matrix
$$
\left[\begin{matrix} t -z & x + iy \cr x - iy & t +z \cr \end{matrix}\right].
$$
It is trivial to check that the
events $(x_1 , y_1 , z_1 , t_1)$
and $(x_2 , y_2 , z_2 , t_2)$ are lightlike if and only if
$$
\det \left( \left[\begin{matrix} t_1 -z_1 & x_1 + iy_1 \cr x_1 - iy_1 & t_1 +z_1 \cr \end{matrix} \right]
- \left[\begin{matrix} t_2 -z_2 & x_2 + iy_2 \cr x_2 - iy_2 & t_2 +z_2 \cr \end{matrix}\right]
\right) = 0
$$
(note that here we have already assumed that $c=1$).
The determinant of a $2\times 2$ matrix is zero 
if and only if this matrix is of rank at most one. Thus,
two 
spacetime events are lightlike if and only if the corresponding $2\times 2$ hermitian
matrices are coherent. In other words, Aleksandrov's theorem can be reformulated as the description of the general form of bijective maps on the space of $2 \times 2$ hermitian matrices that preserve adjacency (in the presence of the bijectivity assumption the property of preserving adjacency is  equivalent to the property of preserving coherency).

An optimal version of the fundamental theorem of chronogeometry has been recently proved in \cite{MoS2}. We have shown that the fundamental theorem of chronogeometry is closely related to the problem of determining the general form of order embeddings on the space of $2 \times 2$ hermitian matrices.  Since $S_2$ can be considered as a real analogue of the space of $2 \times 2$ complex hermitian matrices, it is  not surprising that several ideas used in our proofs are related to some proof techniques in \cite{MoS2}.
 
\item 
In this paper we have described the general form of order embeddings of matrix domains. Let us explain why it is natural to assume in our main result that order embeddings are defined on sets that are connected and open. If, for example we take any order embedding $\phi : (0,I) \cup (2I,3I)$ and denote the restrictions of $\phi$ to $(0,I)$ and $(2I,3I)$ by $\phi_1$ and $\phi_2$, respectively, then we obviously have $\phi_1 (X) \le \phi_2 (Y)$ for every $X \in (0,I)$ and every $Y \in (2I,3I)$, but once this condition is fulfilled the behaviour of $\phi_1$ is completely unrelated to that of $\phi_2$. In order to demonstrate that the openess assumption is natural in our investigations we consider
an easy example of a non-continuous order emebedding of $[0,I]$ into $S_n$.  Let $E,F \in S_n$ with $E \le 0$ and $F \ge I$. Then the map  $\phi  : [0,I] \to S_n$ given by $\phi (X) = X$, $X \in [0,I] \setminus \{0, I \}$, $\phi (0) = E$, $\phi (I) = F$, is obvioulsy an order embedding. 

\item
The above example leads to the next natural question.
We already know that the structure of order embeddings of $[0,I]$ can be nicely described under the additional assumption of continuity at $0$ and $I$. It si then natural to ask what happens in the absence of the assumption of the continuity at the end points. 
A possible conjecture would be that we have a nice behaviour on $[0,I] \setminus \{0, I \}$, while $\phi (0)$ and $\phi(I)$ can be arbitrary matrices that are below and above the image $\phi ( [0,I] \setminus \{0, I \} )$, respectively. As we shall see, this conjecture holds true.

\begin{theorem}\label{bzgf}
Let $n \ge 2$ be an integer and $\phi  : [0,I] \to S_n$ an order emebedding. Then there exist $A,B \in S_n$ and an invertible real $n \times n$ matrix $T$ such that $[0,I] \subset \U_A$ and
$$
\phi (X) = T (XA + I)^{-1} XT^t +B, \ \ \ X \in [0,I] \setminus \{0, I \}.
$$
\end{theorem}

\begin{proof}
We first observe that we can assume with no loss of generality that $\phi (0) =0$. And then, of course, we want to show that $\phi$ is of the desired form with $B=0$. We consider the restriction $\phi'$ of $\phi$ to $(0,I)$. By Corollary \ref{glavposl} there exists a unique maximal extension $\psi: \mathcal{W} \to S_n$ of order embedding  $\phi' : (0,I) \to S_n$. Denote $m = \max \{ \| \phi (0) \| , \| \phi(I) \| \}$. For every
$X \in [0,I]$ we have
$$
-m I \le \phi (0) \le \phi (X) \le \phi (I) \le m I,
$$
and therefore, $\| \phi (X) \| \le m$ for every $X \in  [0,I]$. Applying Corollary \ref{glavposl} once more we conclude that $ [0,I] \subset \mathcal{W}$. We need to verify that $\phi (X) = \psi (X)$ for every $X  \in [0,I] \setminus \{0, I \}$. By the definition of $\psi$ we already have $\phi (X) = \psi (X)$ for every $X \in (0,I)$.

Let $X \in  [0,I] \setminus \{ I \}$ such that $1$ is an eigenvalue of $X$ but $0$ is not an eigenvalue of $X$. Then we can find a positive real number $a$ such that $aI < X$. We know that for every $P \in P_{n}^1$ there exists a rank one projection $P'$ such that for every real $s$, $s \in [0, 1-a)$, we have
$$
\phi (aI + sP) \in \{ \phi (aI) + tP' \, : \, t \ge 0 \} .
$$
We know that the map $P \mapsto P'$ is a homeomorphism of $P_{n}^1$.
Let $R$ be the orthogonal projection onto the eigenspace of $X$ corresponding to the eigenvalue $1$ and denote by $\mathcal{P} \subset P_{n}^1$ the set of all rank one projections $P$ such that $P\not\le R$. Then 
$\alpha (aI, P; X) < 1-a$ for every $P \in \mathcal{P}$. It follows that 
$$
\alpha ( \phi (aI), P' ; \phi (X)) = \alpha ( \phi (aI), P' ; \psi (X))
$$
for every $P \in \mathcal{P}$. By Corollary \ref{equal} we have $\phi (X) = \psi (X)$.

In the same way we prove that $\phi (X) = \psi (X)$ for every $X \in  [0,I] \setminus \{ 0 \}$ with the property that $0$ is an eigenvalue of $X$ and $1$ is not an eigenvalue $X$.

Let finally $X$ be an element of $ [0,I]$ such that both $0$ and $1$ are eigenvalues of $X$. Let $R_0$ and $R_1$ be the orthogonal  projections on the eigenspaces corresponding to the eigenvalues $0$ and $1$, respectively. We know that
for every positive integer $k \ge 2$ we have
$$
\phi( X - {1 \over k} R_1 ) =  \psi( X - {1 \over k} R_1 ) \le \phi(X), \psi(X) \le  \phi( X + {1 \over k} R_0 ) =  \psi( X + {1 \over k} R_0 )
$$
and by the continuity of $\psi$ we also know that
$$
\lim  \psi( X - {1 \over k} R_1 ) = \lim \psi( X + {1 \over k} R_0 ) = \psi (X),
$$
and therofore, $\phi(X) = \psi(X)$ in this case, as well.
\end{proof}

\item
Theorem \ref{bzgf} describes the general form of order embeddings $\phi  : [0,I] \to S_n$. The main result in \cite{Sem} describes order automorphisms of $ [0,I]$ onto itself. It covers also the infinite-dimensional case. In the finite-dimensional case we can describe order embeddings of $ [0,I]$ without the bijectivity assumption, while it turns out that in the infinite-dimensional the bijectivity assumption is indispensable. Actually, this difference between the finite-dimensional and the  infinite-dimensional case is quite usual when dealing with linear or general preservers. It is therefore not surprising that the main idea of the proof in this paper is essentially different from the main idea in \cite{Sem}. But of course, a few technical lemmas are quite similar.

\end{enumerate}


\begin{thebibliography}{99}

\bibitem{Al1}A.D. Alexandrov, On Lorentz transformations, {\em Uspehi Mat. Nauk} \textbf{5} (1950), 187.

\bibitem{Al2}A.D. Alexandrov, A contribution to chronogeometry, {\em Canad. J. Math.} \textbf{19} (1967), 1119--1128.

\bibitem{Al3}A.D. Alexandrov, Mappings of spaces with families of cones and spacetime transformations, {\em Annali di Matematica} \textbf{103} (1975), 229--257.


\bibitem{BG}P. Busch and S.P. Gudder, Effects as functions of projective Hilbert space, {\em Lett. Math. Phys.} {\bf 47} (1999), 329--337.

\bibitem{Cho}W.-L. Chow, On the geometry of algebraic homogeneous spaces,
{\em Ann. Math.} {\bf 50} (1949), 32-67.

\bibitem{Fa}C.-A. Faure, An elementary proof of the fundamental theorem of projective
geometry, {\em Geom. Dedicata} {\bf 90} (2002), 145--151.

\bibitem{Hu1}L.K. Hua, Geometries of matrices I. Generalizations of von Staudt's
theorem, {\em Trans. Amer. Math. Soc.} {\bf 57} (1945), 441--481.


\bibitem{Hu2}L.K. Hua, Geometries of matrices ${\rm I}_1$. Arithmetical
construction, {\em Trans. Amer. Math. Soc.} {\bf 57} (1945), 482--490.

\bibitem{Hu3}L.K. Hua, Geometries of symmetric matrices over the real field I,
{\em Dokl. Akad. Nauk. SSSR} {\bf 53} (1946), 95--97.

\bibitem{Hu4}L.K. Hua, Geometries of symmetric matrices over the real field II,
{\em Dokl. Akad. Nauk. SSSR} {\bf 53} (1946), 195--196.

\bibitem{Hu5}L.K. Hua, Geometries of matrices II. Study of involutions in the geometry
of symmetric matrices,
{\em Trans. Amer. Math. Soc.} {\bf 61} (1947), 193--228.

\bibitem{Hu6}L.K. Hua, Geometries of matrices III. Fundamental theorems in the geometries
of symmetric matrices,
{\em Trans. Amer. Math. Soc.} {\bf 61} (1947), 229--255.

\bibitem{Hu7}L.K. Hua, Geometry of symmetric matrices over any field with characteristic
other than two,
{\em Ann. Math.} {\bf 50} (1949), 8--31.

\bibitem{Hu8}L.K. Hua, A theorem on matrices over a sfield and its applications,
{\em Acta Math. Sinica} {\bf 1} (1951), 109--163.

\bibitem{HuS} W.-l. Huang and P. \v Semrl, Adjacency preserving maps on hermitian matrices, {\em Canad. J. Math.} {\bf 60} (2008), 1050--1066.


\bibitem{Mo1}L. Moln\' ar,  Order-automorphisms of the set of bounded observables, {\em J. Math. Phys.} {\bf 42} (2001), 5904--5909. 


\bibitem{MoS}M. Mori and P. \v Semrl, Loewner's theorem for maps on operator domains, {\em Canad. J. Math.} {\bf 75} (2023), 912--944. 

\bibitem{MoS2}M. Mori and P. \v Semrl, Optimal version of the fundamental theorem of chronogeometry, {\em Adv. Math.} {\bf 480}  (2025), 110528, 85pp.

\bibitem{Se0}P. \v Semrl, Comparability preserving maps on bounded observables, {\em Integral Equations Operator Theory} {\bf 62} (2008), 441--454. 

\bibitem{Se3}P. \v Semrl, The optimal version of Hua's fundamental theorem of geometry of rectangular matrices,  {\em Mem. Amer. Math. Soc.} \bf 232 \rm (2014), 74pp.


\bibitem{Sem}P. \v Semrl, Order automorphisms of effect algebras, accepted in {\em Illinois J. Math.}

\bibitem{Wan}Z.-X. Wan, {\em Geometry of matrices}, World Scientific, Singapore, 1996.




\end{thebibliography}
\end{document}